%% file: AmorphicComplexity_arXiv_v2.tex
\documentclass[10pt,oneside,a4paper]{scrartcl}

\usepackage[left=3cm,right=4cm,top=3cm,bottom=3cm,includeheadfoot]{geometry}
\usepackage[english]{babel}
\usepackage[utf8]{inputenc}
\usepackage{fancyhdr}
\usepackage{amsmath, amssymb, latexsym,
   epsfig, rotating, fancyhdr, amsthm, pifont}
\usepackage{amsmath}
\usepackage{verbatim} 
\usepackage{mathtools}
\usepackage{subfig}
\usepackage{graphicx}
\input bibstandard
\newtheorem{theorem}{Theorem}[section]
\newtheorem{corollary}[theorem]{Corollary}
\newtheorem{proposition}[theorem]{Proposition}
\newtheorem{lemma}[theorem]{Lemma}

\theoremstyle{definition}

\newtheorem*{example*}{Example}
\newtheorem{remark}[theorem]{Remark}
\newtheorem{remarks}[theorem]{Remarks}

\newcommand{\I}{\ensuremath{\mathbb{I}}}
\DeclareMathOperator*{\Times}{\mathop{\raisebox{-.5ex}{\hbox{\huge{$\times$}}}}}

\newcommand{\Bcal}{\ensuremath{\mathcal{B}}}

\newcommand{\abs}[1]{\left|#1\right|}

\newcommand{\countsep}[1]{S_{#1}}
\newcommand{\fsc}{\mathrm{{ac}}}
\newcommand{\Sep}{\ensuremath{\mathrm{Sep}}}
\newcommand{\Span}{\ensuremath{\mathrm{Span}}}
\newcommand{\Per}{\ensuremath{\mathrm{Per}}}

\newcommand{\oac}{\ensuremath{\overline{\mathrm{ac}}}}
\newcommand{\uac}{\ensuremath{\underline{\mathrm{ac}}}}
\newcommand{\Dim}{\ensuremath{\mathrm{Dim}}}

\renewcommand{\phi}{\varphi}
\newcommand{\lam}{\lambda}

\title{\Large\textsc{Amorphic complexity}}
\author{\normalsize
	G. Fuhrmann
		\thanks{Department of Mathematics, TU Dresden, Germany. Email: 
		{\tt gabrielfuhrmann@googlemail.com}}
    \and \normalsize M. Gr\"oger
		\thanks{Department of Mathematics, Universit\"at Bremen,
		Germany. Email: {\tt groeger@math.uni-bremen.de}} 		
	\and \normalsize T. J\"ager
		\thanks{Department of Mathematics, TU Dresden, Germany. Email: 
		{\tt Tobias.Oertel-Jaeger@tu-dresden.de}}
}

\begin{document}
\renewcommand\dagger{**}
\maketitle

\begin{abstract}\small 
	We introduce amorphic complexity as a new topological invariant that
	measures the complexity of dynamical systems in the regime of zero entropy.
	Its main purpose is to detect the very onset of disorder in the asymptotic
	behaviour.
	For instance, it gives positive value to Denjoy examples on the circle and
	Sturmian subshifts, while being zero for all isometries and Morse-Smale
	systems.

	After discussing basic properties and examples, we show that amorphic
	complexity and the underlying asymptotic separation numbers can be used to
	distinguish almost automorphic minimal systems from equicontinuous ones.
	For symbolic systems, amorphic complexity equals the box dimension of the
	associated Besicovitch space.
	In this context, we concentrate on regular Toeplitz flows and give a
	detailed description of the relation to the scaling behaviour of the
	densities of the $p$-skeletons.
	Finally, we take a look at strange non-chaotic attractors appearing in
	so-called pinched skew product systems.
	Continuous-time systems, more general group actions and the application to
	cut and project quasicrystals will be treated in subsequent work.
\end{abstract}

\section{Introduction}
\label{Intro}

The paradigm example of a topological complexity invariant for dynamical systems
is topological entropy.
It measures the exponential growth, in time, of orbits distinguishable at finite
precision and can be used to compare the complexity of dynamical systems defined
on arbitrary compact metric spaces.
Moreover, it is central to the powerful machinery of thermodynamic formalism.
There are, however, two situations where entropy does not provide very much
information, namely when it is either zero or infinite. 
In the latter case, mean topological dimension has been identified as a suitable
substitute. 
Its theoretical significance is demonstrated, for example, by the fact that zero
mean dimension is one of the few dynamical consequences of unique ergodicity
\cite{LindenstraussWeiss2000MeanDimension}.

Our focus here lies on the zero entropy regime, and in particular on the very
onset of dynamical complexity and the break of equicontinuity.
We are looking for a dynamically defined positive real-valued quantity which
\romanlist
	\item is an invariant of topological conjugacy and has other good properties;
	\item gives value zero to isometries and Morse-Smale systems;
	\item is able to detect, as test cases, the complexity inherent in the
		dynamics of Sturmian shifts or Denjoy homeomorphisms on the circle,
		by taking positive values for such systems.
\listend

There exist several concepts to describe the complexity of systems in the
zero entropy regime (see, for example, \cite{Misiurewicz1981,Smital1986,
MisiurewiczSmital1988,KolyadaSharkovsky1991,Carvalho1997,
Ferenczi1997MeasureTheoreticComplexity,KatokThouvenot1997SlowEntropy,Ferenczi1999,
BlanchardHostMaas2000TopologicalComplexity,
HasselblattKatok2002HandbookPrincipalStructures, FerencziPark2007,HuangParkYe2007,
HuangYe2009,ChengLi2010,DouHuangPark2011,Marco2013,KongChen2014}).
Some of them have properties that may be considered as shortcomings, although
this partly depends on the viewpoint and the particular purpose one has in mind.
To be more precise, let us consider one example of a standard approach to measure
the complexity of zero entropy systems, namely, the (modified) power entropy
\cite{HasselblattKatok2002HandbookPrincipalStructures}.
In the context of tiling spaces and minimal symbolic subshifts, power entropy is
more commonly known as polynomial word complexity and presents a well-established
tool to  describe the complexity of aperiodic sequences.
However, it turns out that power entropy gives positive values to Morse-Smale
systems, whereas modified power entropy is too coarse to distinguish Sturmian
subshifts or Denjoy examples from irrational rotations.

We are thus taking an alternative and complementary direction, which leads us to
define the notions of \emph{asymptotic separation numbers} and
\emph{amorphic  complexity}.
Those are based on an asymptotic notion of separation, which is the main
qualitative difference to the previous two concepts, since the latter rely in
their definition on the classical Bowen-Dinaburg/Hamming metrics which consider
only finite time-scales.
As a consequence, ergodic theorems can be applied in a more or less direct way
to compute or estimate amorphic complexity in many situations.
In order to fix ideas, we concentrate on the dynamics of continuous maps defined
on metric spaces.
Continuous-time systems and more general group actions will be treated in
future work.\medskip

Let $(X,d)$ be a metric space and $f:X\to X$.  Given $x,y\in X$, $\delta>0$,
$\nu\in(0,1]$ and $n\in\N$ we let
\begin{equation}\label{e.separation_count}
	\countsep{n}(f,\delta,x,y)\ := \ \#\left\{0\leq k<n\;|\;
		d(f^k(x),f^k(y))\geq\delta\right\}.
\end{equation}
We say that $x$ and $y$ are \emph{$(f,\delta,\nu)$-separated} if
\[
	\varlimsup\limits_{n\to\infty}\frac{\countsep{n}(f,\delta,x,y)}{n}\ \geq \ \nu \ .
\]
A subset $S\subseteq X$ is said to be \emph{$(f,\delta,\nu)$-separated} if all
distinct points $x,y\in S$ are $(f,\delta,\nu)$-separated.  
The {\em (asymptotic) separation number} $\Sep(f,\delta,\nu)$, for distance
$\delta>0$ and frequency $\nu\in(0,1]$, is then defined as the largest
cardinality of an $(f,\delta,\nu)$-separated set in $X$.
If these quantities are finite for all $\delta,\nu>0$, we say $f$ has
{\em finite separation numbers}, otherwise we say it has
{\em infinite separation numbers}.
Further, if $\Sep(f,\delta,\nu)$ is uniformly bounded in $\nu$ for all $\delta>0$,
we say that $f$ has {\em bounded separation numbers}, otherwise we say
{\em separation numbers are unbounded}.

These notions provide a first qualitative indication concerning the complexity
of a system.
Roughly spoken, finite but unbounded separation numbers correspond to dynamics
of intermediate complexity, which we are mainly interested in here.
Once a system behaves `chaotically', in the sense of positive entropy or weak
mixing, separation numbers become infinite.
\begin{theorem}
	Suppose $X$ is a compact metric space and $f:X\to X$ is continuous.
	If $f$ has positive topological entropy or is weakly mixing with respect to
	some invariant probability measure $\mu$ with non-trivial support, then it
	has infinite separation numbers.
\end{theorem}
The proof is given in Section~\ref{QualitativeBehaviour}.
Obviously, if $f$ is an isometry or, more generally, equicontinuous, then its
separation numbers are bounded.
Moving away from equicontinuity, one encounters the class of almost automorphic
systems, which are central objects of study in topological dynamics and include
many examples of both theoretical and practical importance
\cite{auslander1988minimal}.
At least in the minimal case, separation numbers are suited to describe this
transition, as the following result shows.
In order to state it, suppose that $(X,d)$ and $(\Xi,\rho)$ are metric spaces
and $f:X\to X$ and $g:\Xi\to\Xi$ are continuous.
We say that $f$ is an extension of $g$ if there exists a continuous onto map
$h:X\to \Xi$ such that $h\circ f=g\circ h$.
The map $f$ is called an {\em almost 1-1 extension} of $g$ if the set 
$\{\xi\in\Xi\mid \#h^{-1}(\xi)=1\}$ is dense in $\Xi$.
In the case that $g$ is minimal, this condition can be replaced by the weaker
assumption that there exists one $\xi\in\Xi$ with $\#h^{-1}(\xi)=1$.
We further say that $f$ is an {\em almost sure 1-1 extension} if the set
$\{\xi\in \Xi\mid \#h^{-1}(\xi)>1\}$ has measure zero with respect to every
$g$-invariant probability measure $\mu$ on $\Xi$.\foot{Note that if $g$ is
equicontinuous and minimal, then it is uniquely ergodic.
Hence, there is only one measure to consider in this case.}
Due to Veech's Structure Theorem \cite{veech1965almost}, almost automorphic
minimal systems can be defined as almost 1-1 extensions of equicontinuous
minimal systems.
\begin{theorem} \label{t.automorphic_systems}
	Suppose $X$ is a compact metric space and $f:X\to X$ is a homeomorphism.
	\alphlist
		\item If $f$ is minimal and almost automorphic, but not equicontinuous,
			then $f$ has unbounded separation numbers.
		\item If $f$ is an almost sure 1-1 extension of an equicontinuous system,
			then $f$ has finite separation numbers.
	\listend
\end{theorem}
Again, the proof is given in Section~\ref{QualitativeBehaviour}.
Two examples for case (b) discussed below are regular Toeplitz flows and Delone
dynamical systems arising from cut and project quasicrystals.
We refer to \cite{LiTuYe2014,DownarowiczGlasner2015} for some recent progress
on extensions of minimal equicontinuous systems.\medskip

In order to obtain quantitative information, we proceed to study the scaling
behaviour of separation numbers as the separation frequency $\nu$ goes to zero.
In principle, one may consider arbitrary growth rates (see
Section~\ref{sec_definitions}).
However, as the examples we discuss all indicate, it is polynomial growth which
is the most relevant.
Given $\delta>0$, we let
\begin{equation}\label{d.delta_amorphic_complexity}
	\uac(f,\delta)\ := \ \varliminf_{\nu\to 0}
		\frac{\log \Sep(f,\delta,\nu)}{-\log \nu} \  ,\qquad
	\oac(f,\delta) \ := \ \varlimsup_{\nu\to 0}
			\frac{\log\Sep(f,\delta,\nu)}{-\log \nu}
\end{equation}
and define the {\em lower}, respectively {\em upper amorphic complexity of $f$} as 
\begin{equation}\label{e.amorphic_complexity}
	\uac(f) \ := \ \sup_{\delta>0} \uac(f,\delta) \eqand
	\oac(f) \ := \ \sup_{\delta>0} \oac(f,\delta) \ . 
\end{equation}
If both values coincide, $\fsc(f):=\uac(f)=\oac(f)$ is called the {\em amorphic
complexity of $f$}.
We note once  more that the main difference to the notion of (modified) power
entropy is the fact that we use an asymptotic concept of separation, and the
scaling  behaviour that is measured is not the one with respect to time, but
that with respect to the separation frequency.
Somewhat surprisingly, this makes amorphic complexity quite well-accessible to
rigorous computations and estimates.
The reason is that separation frequencies often correspond to certain ergodic
averages or visiting frequencies, which can be determined by the application of
ergodic theorems.
We have the following basic properties.
\begin{proposition}
	Suppose $X,Y$ are compact metric spaces and $f:X\to X,\ g:Y\to Y$
	continuous maps.
	Then the following statements hold.
	\alphlist
		\item {\em Factor relation:} If $g$ is a factor of $f$, then
			$\oac(f)\geq\oac(g)$ and  $\uac(f)\geq \uac(g)$.
			In particular, amorphic complexity is an invariant of topological
			conjugacy.
		\item {\em Power invariance:} For all $m\in\N$ we have
			$\oac(f^m)=\oac(f)$ and $\uac(f^m)=\uac(f)$. 
		\item {\em Product formula:} If upper and lower amorphic complexity
			coincide for both $f$ and $g$, then the same holds for $f\times g$
			and we have $\fsc(f\times g)=\fsc(f)+\fsc(g)$.
			Otherwise, we have $\oac(f\times g)\leq \oac(f)+\oac(g)$ and
			$\uac(f\times g)\geq \uac(f)+\uac(g)$.
		\item {\em Commutation invariance:}
			$\oac(f\circ g)=\oac(g\circ f)$ and $\uac(f\circ g)=\uac(g\circ f)$.
	\listend
\end{proposition}

As the power invariance indicates, amorphic complexity behaves quite different
from topological entropy in some aspects.
In this context, it should also be noted that no variational principle can be
expected for amorphic complexity. 
This is a direct consequence of requirement (iii) above, which is met by
amorphic complexity (see Proposition \ref{p.ac_morse_iso_Sturmian_Denjoy}). 
The reason is that since Sturmian subshifts, Denjoy examples and irrational
rotations are uniquely ergodic and measure-theoretically isomorphic, they cannot
be distinguished on a measure-theoretic level. 
Hence, no reasonable analogue to the variational principle of topological
entropy can be satisfied.
\begin{proposition}\label{p.ac_morse_iso_Sturmian_Denjoy}
	Amorphic complexity is zero for all isometries and Morse-Smale systems, but
	equals one for Sturmian subshifts and Denjoy examples on the circle.
\end{proposition}
The proof is given in Sections \ref{Isometries} and \ref {BasicExamples}.
By means of some elementary examples in Section~\ref{PowerEntropy}, we will also
demonstrate that no direct relations -- in terms of inequalities -- exist
between amorphic complexity and the notions of power entropy and modified power
entropy.\medskip

The arguments in the proof of Theorem~\ref{t.automorphic_systems} can be
quantified, at least to some extent, to obtain an upper bound on amorphic
complexity for minimal almost sure 1-1 extensions of isometries.
In rough terms, the result reads as follows.
Details will be given in Section~\ref{QuantitativeAutomorphic}.
By $\overline{\Dim}_B(A)$ we denote the upper box dimension of a totally bounded
subset $A$ of a metric space.
\begin{theorem}\label{t.automorphic_quantitative_intro}
	Suppose $X$ and $\Xi$ are compact metric spaces and $f:X\to X$ is an almost
	sure 1-1 extension of a minimal isometry $g:\Xi\to\Xi$ with factor map $h$.
	Further, assume that the upper box dimension of $\Xi$ is finite and strictly
	positive.
	Then
	\begin{equation} \label{e.automorphic_estimate}
		\oac(f) \ \leq \ \frac{\gamma(h)\cdot\overline\Dim_B(\Xi)}
		{\overline\Dim_B(\Xi)-\sup_{\delta>0}\overline\Dim_B(E_\delta)} \ ,
	\end{equation}
	where $E_\delta=\{\xi\in\Xi\mid \diam(h^{-1}(\xi))\geq \delta\}$ and
	$\gamma(h)$ is a scaling factor depending on the local properties of the
	factor map $h$.
\end{theorem}
The proof is given in Section~\ref{QuantitativeAutomorphic}.
It should be mentioned, at least according to our current understanding, that
this result is of rather abstract nature.
The reason is the fact that the scaling factor $\gamma(h)$, defined by
\eqref{e.gamma_def}, seems to be difficult to determine in specific examples.
However, it turns out that in many cases direct methods can be used instead to
obtain improved explicit estimates.\medskip

In this direction, we will first investigate {\em regular Toeplitz flows} in
Section~\ref{RegularToeplitzFlows}.
Given a finite alphabet $A$, a sequence $\omega=(\omega_k)_{k\in\I}\in A^\I$ with
$\I=\N_0$ or $\Z$ is called {\em Toeplitz} if for all $k\in\I$ there exists
$p\in\N$ such that $\omega_{k+p\ell}=\omega_k$ for all $\ell\in\N$.
In other words, every symbol in a Toeplitz sequence occurs periodically.
Thus, if we let $\Per(p,\omega)=\{k\in\I\mid \omega_{k+p\ell}=\omega_k \ 
	\textrm{for all}\ \ell\in\N\}$, then $\bigcup_{p\in\N}\Per(p,\omega)=\I$.
By $D(p)=\#(\Per(p,\omega)\cap[0,p-1])/p$ we denote the density of the
$p$-periodic positions. If $\lim_{p\to\infty} D(p)=1$, then the Toeplitz
sequence is called {\em regular}.
A well-known example of a regular Toeplitz sequence is the paperfolding sequence, also
known as the dragon curve sequence \cite{AlloucheBacher1992}.

We call a sequence $(p_\ell)_{\ell\in\N}$ of integers such that $p_{\ell+1}$ is a
multiple of $p_\ell$ for all $\ell\in\N$ and
$\bigcup_{\ell\in\N}\Per(p_\ell,\omega)=\I$ a {\em weak periodic structure}
for $\omega$. More details are given in Section~\ref{RegularToeplitzFlows}.
We denote the shift orbit closure of $\omega$ by $\Sigma_\omega$ such that
$(\Sigma_\omega,\sigma)$ is the subshift generated by $\omega$.
\begin{theorem} \label{t.toeplitz}
	Suppose $\omega$ is a non-periodic regular Toeplitz sequence with weak
	periodic structure $(p_\ell)_{\ell\in\N}$.
	Then
	\[
		\oac\left(\left.\sigma\right|_{\Sigma_\omega}\right) \ \leq \ 
		\varlimsup_{\ell\to\infty}\frac{\log p_{\ell+1}}{\log(1-D(p_\ell))} \ .
	\]
\end{theorem}
In Section~\ref{RegularToeplitzFlows}, we demonstrate by means of examples that
this estimate is sharp and that a dense set of values in $[1,\infty)$ is
attained (Theorem~\ref{t.toeplitz_sharpbound_examples} and
Corollary~\ref{c.toeplitz_densevalues}).
A more comprehensive treatment of amorphic complexity for symbolic systems of
intermediate complexity (zero topological entropy) will be given in
\cite{FG2014} (see also Section \ref{BesicovitchSpace}).

In Section~\ref{PinchedSystems}, we take a closer look at strange
non-chaotic attractors (SNA) appearing in so-called pinched skew products.
The latter are known as paradigm examples for the occurrence of SNA
\cite{grebogi/ott/pelikan/yorke:1984,keller:1996,glendinning/jaeger/keller:2006}.
In this case technical issues prevent a straightforward computation of
amorphic complexity, and we are only able to apply a modified version of the
concept.
However, since the attempt to distinguish SNA and smooth (non-strange) attractors
in skew product systems by means of topological invariants has been the origin
of our investigations, it seemed important to include these findings.\medskip

Finally, we also want to include a research announcement of a result from the
forthcoming paper \cite{FGJ2014AmorphicComplexityQuasicrystals}, which
fits well into the above discussion.
Suppose $\tilde L$ is a cocompact discrete subgroup of $\R^m\times\R^D$ such that
$\pi_1:\tilde L\to \R^m$ is injective and $\pi_2:\tilde L\to \R^D$ has dense image.
Further, assume that $W\subset \R^D$ is compact and satisfies $W=\overline{\inte(W)}$.
The pair $(\tilde L,W)$ is called a {\em cut and project scheme} and defines a
Delone subset $\Lambda(W):=\pi_1\big((\R^m\times W)\cap \tilde L\big)$ of $\R^m$.
A natural $\R^m$-action on the space of Delone sets in $\R^m$ is given by
$(t,\Lambda)\mapsto \Lambda-t$.
Taking the orbit closure $\Omega(\Lambda(W)) := \overline{\{\Lambda(W)-t\mid t\in\R^m\}}$
of $\Lambda(W)$, in a suitable topology, we obtain a {\em Delone dynamical system}
$(\Omega(\Lambda(W)),\R^m)$ whose dynamical properties are closely related to the
geometry of the Delone set $\Lambda(W)$.
We refer to
\cite{Schlottmann1999GeneralizedModelSets,Moody2000ModelSetsSurvey,
	LagariasPleasants2003RepetitiveDeloneSets,BaakeLenzMoody2007Characterization}
and references therein for further details. For the amorphic complexity, adapted
to general actions of amenable groups, we obtain
\begin{theorem}[\cite{FGJ2014AmorphicComplexityQuasicrystals}]
	Suppose $(\tilde L,W)$ is a cut and project scheme in $\R^m\times \R^D$ and
	$(\Omega(\Lambda(W)),\R^m)$ is the associated Delone dynamical system.
	Then
	\begin{equation}\label{e.quasicrystal_estimate}
		\oac(\Omega(\Lambda(W)),\R^m) \ \leq \ \frac{D}{D-\overline\Dim_B(\partial W)} \ . 
	\end{equation}
\end{theorem}
As in the case of regular Toeplitz flows, it can be demonstrated by means of
examples that this estimate is sharp.
At the same time, equality does not always hold. 

It is well-known that under the above assumptions the dynamical system
$(\Omega(\Lambda(W)),\R^m)$ is an almost 1-1 extension of a minimal and isometric
$\R^m$-action on a $D$-dimensional torus.
Moreover, it turns out that with the notions of
Theorem~\ref{t.automorphic_quantitative_intro} we have
$\overline\Dim_B(\partial W)=\overline\Dim_B(E_\delta)$ for all $\delta>0$.
Thus, \eqref{e.quasicrystal_estimate} can be interpreted as a special case of
\eqref{e.automorphic_estimate}, with $\gamma(h)=1$.
However, as we have mentioned, the proof is independent and based on more direct
arguments.\medskip

\noindent{\bf Acknowledgments.} The above results were first presented during
the conference `Complexity and Dimension Theory of Skew Product Systems' in
Vienna in September 2013, and we would like to thank the organisers
Henk Bruin and Roland Zweim\"uller for creating this opportunity as well as the
Erwin-Schr\"odinger-Institute for its hospitality and the superb conditions
provided during the event.
T.\ J.\ also thanks the organisers of the `Dynamics and Numbers activity' (MPIM
Bonn, June--July 2014), during which this work was finalized.
We are indebted to Tomasz Downarowicz for his thoughtful remarks, and in
particular for suggesting the study of Toeplitz systems.
All authors acknowledge support of the German Research Council (Emmy Noether
Grant Ja 1721/2-1) and M.\ G.\ has been supported by a doctoral scholarship of
the `Studienstiftung des deutschen Volkes'.

\section{Qualitative behaviour of asymptotic separation numbers}
\label{QualitativeBehaviour}
Let $(X,\Bcal,\mu)$ be a probability space and let $\mu$ be invariant
with respect to the measurable map $f:X\to X$.
For the definition of ergodic and weak-mixing measures, respectively, see for
example \cite{Walters1982}.
\begin{theorem}[{\cite[Theorem 4.10.6]{BrinStuck2002}}]
	\label{theorem_equivalence_weak_mixing_ergodicity}
	The following statements are equivalent
	\begin{enumerate}
		\item[(a)] $\mu$ is weak-mixing with respect to $f$,
		\item[(b)] $\mu^m=\Times_{k=1}^{m}\mu$ is ergodic with respect to
			$\Times_{k=1}^m f$ for all $m\geq 2$.
	\end{enumerate}
\end{theorem}
\begin{theorem}
	Let $(X,d)$ be a metric space.
	Suppose $f:X\to X$ is Borel measurable and $\mu$ is a Borel probability
	measure invariant under $f$.
	Furthermore, assume that $\mu$ is weak-mixing with respect to $f$ and its
	support is not a single point.
    Then $f$ has infinite separation numbers.
\end{theorem}
\begin{proof}
	For each $\delta>0$ we define the function $h_\delta:X^2\to\{0,1\}$ as
	$h_\delta(z,w):=\Theta(d(z,w)-\delta)$ where $\Theta:\R\to\{0,1\}$ is the
	Heaviside step function.
	Note that
	\[
		\frac{1}{n}\countsep{n}(f,\delta,x,y)
		\ = \ \frac{1}{n}\sum\limits_{k=0}^{n-1}h_\delta\big(f^k(x),f^k(y)\big) \ .
	\]
	Since $\mu$ is not supported on a single point, we can find $\delta_0>0$ and
	$\nu_0>0$ such that for all $\delta\leq\delta_0$ we have
	\begin{align}\label{thm_weak_mixing_heaviside_prop}
          \int h_\delta d\mu^2\ \geq \ \nu_0 \ .
	\end{align}
	(Note that $\int h_{\delta'}d\mu^2\geq \int h_\delta d\mu^2$ for $\delta'\leq\delta$.) 
	Fix $\delta\in(0,\delta_0]$, $\nu\in (0,\nu_0]$ and let
	\begin{align}
		\label{pairwise_frequency _separation}
		\phi_m:X^m\to\R^{m(m-1)/2}:
		\begin{pmatrix}	x_1\\ x_2\\ \vdots\\ x_m \end{pmatrix}
		\mapsto\lim\limits_{n\to\infty}\frac{1}{n}\sum\limits_{k=0}^{n-1}
		\begin{pmatrix}	
			h_\delta(f^k(x_1),f^k(x_2))\\
			h_\delta(f^k(x_1),f^k(x_3))\\ \vdots\\
			h_\delta(f^k(x_{m-1}),f^k(x_m))
		\end{pmatrix}
	\end{align}
	for each $m\geq 2$.
	Since $h_\delta$ is bounded, observe that the functions
    $(x_1,\dots,x_m)\mapsto h_\delta(x_i,x_j)$ with $1\leq i<j\leq m$ are in
	$L^1(\mu^m)$.  By ergodicity of $\mu^m$, the limits in
	\eqref{pairwise_frequency _separation} exist $\mu^m$-almost everywhere.
	Further, $\phi_m$ is $\mu^m$-almost surely constant and all its entries
	are different from zero, since we have
	\begin{eqnarray*}
          \lefteqn{\lim\limits_{n\to\infty}\frac{1}{n}\sum\limits_{k=0}^{n-1}
            h_\delta\left(f^k(x_i),f^k(x_j)\right)
            \ = \ \int\limits_{X^m}h_\delta(x_i,x_j)d\mu^m(x_1,\dots,x_m) } \\
          &= & \int\limits_{X^2}h_\delta(x_i,x_j)d\mu(x_i)d\mu(x_j)
          \ \geq \ \nu_0 \ > \ 0 \hspace{8eM} 
	\end{eqnarray*}
	for $1\leq i<j\leq m$ by \eqref{thm_weak_mixing_heaviside_prop}.
	Thus, the above implies that for each $m\in\N$ there exist at least $m$ points
	that are pairwise $(f,\delta,\nu)$-separated, so that
	\[
		\Sep(f,\delta,\nu) \ \geq \ m \ .
	\]
	Since $m$ was arbitrary and the pair $(\delta,\nu)$ is fixed, we get that
	$\Sep(f,\delta,\nu)$ is infinite.
\end{proof}\smallskip

The analogous statement for maps with positive topological entropy is a direct
consequence of a result of Downarowicz in \cite{Downarowicz2014}.
In order to state it, we say that two points $x$ and $y$ in a metric space
$(X,d)$ are \emph{DC2-scrambled with respect to $f$} if the following two
conditions are fulfilled
\begin{eqnarray}
	\forall\delta>0&:&\varlimsup\limits_{n\to\infty}
		\frac{\#\left\{0\leq k<n\;|\;d(f^k(x),f^k(y))<\delta_{\phantom 0}\right\}}
			{n} \ = \ 1 \ ,\nonumber\\
	\exists\delta_0>0&:&\varliminf\limits_{n\to\infty}
		\frac{\#\left\{0\leq k<n\;|\;d(f^k(x),f^k(y))<\delta_0\right\}}
			{n}\ < \ 1 \ .\label{e.scrambling2}
\end{eqnarray}
Furthermore, we say that a subset $S\subseteq X$ is \emph{DC2-scrambled} if any
pair $x,y\in S$ with $x\neq y$ is DC2-scrambled. 
The set $S$ is called \emph{uniformly DC2-scrambled} if the $\delta_0$'s and the
lower frequencies in \eqref{e.scrambling2} are uniform for all pairs
$x,y\in S$ with $x\neq y$.  
Now by \cite[Theorem 1.2]{Downarowicz2014}, if $f$ has positive topological
entropy, then there exists an uncountable DC2-scrambled set $S$, and as stated
in \cite[Remark 2]{Downarowicz2014} this set can be chosen uniformly
DC2-scrambled. 
It is then obvious from \eqref{e.scrambling2} that the points in $S$ are pairwise
$(f,\delta,\nu)$-separated for the respective parameters $\delta,\nu>0$, i.e.\ 
$\Sep(f,\delta,\nu)=\infty$. 
Thus, we obtain
\begin{theorem}
	Let $(X,d)$ be a compact metric space.  
	Suppose $f:X\to X$ is a continuous map with positive topological entropy.
	Then $f$ has infinite separation numbers.
\end{theorem}
\smallskip 

We now turn to the opposite direction and aim to show that almost sure 1-1
extensions of equicontinuous systems have finite separation numbers.
In order to do so, we need to introduce some further notions and preliminary statements.
Suppose $(X,d)$ and $(\Xi,\rho)$ are compact metric spaces and $f:X\to X$ is an
extension of $g:\Xi\to \Xi$ with factor map $h:X\to \Xi$.
For $x\in X$, define the {\em fibre of $x$} as $F_x:=h^{-1}(h(x))$.
Denote the collection of fibres by $\cF:=\{F_x \mid x\in X\}$.
Given $\delta>0$, let
\[
	\textstyle\cF_{\delta}\ := \ \{x\in X\;|\;\diam(F_x)\geq\delta\}
	\ = \ \bigcup_{\stackrel{F\ssq\cF}{\diam(F)\geq\delta}} F \ .
\]
Further, let $\cF_{>0}:=\bigcup_{\delta>0}\cF_\delta$, $E_\delta:=h(\cF_\delta)$
and $E:=h(\cF_{>0})$.
Obviously, both $\cF_\delta$ and $E_\delta$ are decreasing in $\delta$.
The next lemma is well-known and we omit the easy proof.
\begin{lemma}\label{lemma_E_delta_closed}
	The set $\cF_\delta$ is closed for all $\delta>0$. 
\end{lemma}

Note that as a direct consequence the sets $\cF_{>0}$, $E_\delta$ and $E$ are
Borel measurable.
The following basic observation will be crucial in the proof of the next theorem.
From now on, we denote by $B_{\eps}(A)$ for $\eps>0$ the open $\eps$-neighborhood of a
subset $A$ of a metric space.

\begin{lemma} \label{l.eta_function}
	For all $\delta>0$ and $\varepsilon>0$ there exists $\eta=\eta_\delta(\eps)>0$
	such that if $x,y\in X$ satisfy $d(x,y)\geq\delta$ and $\rho(h(x),h(y))<\eta$,
	then $h(x)$ and $h(y)$ are contained in $B_\eps(E_\delta)$.
\end{lemma}
\begin{proof}
	Assume for a contradiction that the statement is false.
	Then there are $\delta, \varepsilon>0$ and sequences $(x_k)_{k\in\N}$,
	$(y_k)_{k\in\N}$ in $X$ such that $h(x_k)\notin B_\eps(E_\delta)$ or
	$h(y_k)\notin B_\eps(E_\delta)$ and $d(x_k,y_k)\geq\delta$ for all $k\in\N$,
	but $\rho(h(x_k),h(y_k))\to 0$ as $k\to\infty$.
	By going over to subsequences if necessary, we may assume that 
	$(h(x_k))_{k\in\N}$ lies in $X\backslash B_\eps(E_\delta)$ and that 
	$(x_k)_{k\in\N}$ and $(y_k)_{k\in\N}$ converge.
	Let $x:=\lim_{k\to\infty} x_k$ and $y:=\lim_{k\to\infty} y_k$.
	Then $d(x,y)\geq\delta$ and $h(x) =\kLim h(x_k) \notin B_\eps(E_\delta)$.
	However, $h(x)=h(y)$ and thus $\diam(F_x)=\diam(F_y)\geq\delta$, such that
	$x\in\cF_\delta$, which is the required contradiction.
\end{proof}

\begin{theorem}\label{t.finite_separation_numbers}
	Let $f:X\to X$ be a continuous map. 
	Further, assume that $f$ is an almost sure 1-1 extension of an isometry
	$g:\Xi\to\Xi$.
	Then $f$ has finite separation numbers.
\end{theorem}

Note that this implies Theorem~\ref{t.automorphic_systems}(b), since any
equicontinuous system is an isometry with respect to an equivalent metric.

\begin{proof}
	Denote by $\cM(g)$ the set of all $g$-invariant Borel probability measures on
	$\Xi$.
	Fix $\delta>0$ and $\nu>0$. We claim that since
	$\mu(E_\delta)\leq \mu(E)=0$ for all $\mu\in\cM(g)$, there exists $\eps>0$
	such that
	\begin{equation} \label{e.eps_choice}
		\mu\left(\overline{B_\eps(E_\delta)}\right) \ < \ \nu \quad
		\textrm{for all } \mu\in\cM(g) \ .
	\end{equation}
	Otherwise, it would be possible to find a sequence $\mu_n\in\cM(g)$ with
	$\mu_n\big(\overline{B_{1/n}(E_\delta)}\big)\geq \nu$, which can be
	chosen such that it converges to some $\mu\in\cM(g)$ in the
	weak-$\ast$-topology.
	If $\varphi_m(\xi):=\max\{1-m\cdot d(\xi,B_{1/m}(E_\delta)),0\}$, then we
	have $\int_\Xi\varphi_m\ d\mu_n\geq \nu$ for all $n\geq m$ and hence
	$\int_\Xi\varphi_m\ d\mu\geq \nu$ for all $m\in\N$.
	However, this implies $\mu(E_\delta)\geq \nu$ by dominated convergence,
	contradicting our assumptions. Hence, we may choose $\eps>0$ as in
	\eqref{e.eps_choice}.

	This, in turn, implies that
	\begin{equation}\label{e.visits_to_B_eps}
		\varlimsup_{n\to\infty}\frac{\#\big\{0\leq k<n
			\mid g^k(\xi)\in \overline{B_\eps(E_\delta)}\big\}}{n} \ < \ \nu
	\end{equation}
	for all $\xi\in \Xi$.
	If this was not the case, it would again be possible to construct a
	$g$-invariant measure $\mu$ contradicting \eqref{e.eps_choice}, this time as
	a limit of finite sums $\mu_{\ell}:=\frac{1}{n_{\ell}}\sum_{k=0}^{n_{\ell}-1}
		\delta_{g^k(\xi)}$ of weighted Dirac measures for some $\xi\in\Xi$ that
	does not satisfy \eqref{e.visits_to_B_eps}. (Note that in this situation we
	have $\mu_{\ell}\big(\overline{B_\eps(E_\delta)}\big)\geq\nu$ for all
	$\ell\in\N$, and this inequality carries over to the limit $\mu$ by the
	Portmanteau Theorem.)

	Hence, given any pair $x,y\in X$, the frequency by which both of the iterates
	of $h(x)$ and $h(y)$ visit $\overline{B_\eps(E_\delta)}$ at the same time
	is smaller than $\nu$.
	Together with Lemma \ref{l.eta_function}, this implies that if
	$\rho(h(x),h(y))<\eta_{\delta}(\eps)$, then the points $x$ and $y$ cannot
	be $(f,\delta,\nu)$-separated.
	Thus, if $S\ssq X$ is an $(f,\delta,\nu)$-separated set, then the set $h(S)$
	must be $\eta_\delta(\eps)$-separated (compare Section \ref{BesicovitchSpace})
	with respect to the metric $\rho$.
	By compactness, the maximal cardinality $N$ of an $\eta_\delta(\eps)$-separated
	set in $\Xi$ is bounded. We obtain
	\begin{equation}\label{e.automorphic_Sep_bound}
		\Sep(f,\delta,\nu) \ \leq \ N \ .
	\end{equation}
	Since $\delta>0$ and $\nu>0$ where arbitrary, this completes the proof.
\end{proof}

As immediate consequences, we obtain
\begin{corollary}
	If for all $\delta>0$ the set $E_{\delta}$ is finite and contains no
	periodic point, then $f$ has finite separation numbers.
\end{corollary}
\begin{corollary}
	If $\lim_{n\to\infty}\diam\big(F_{f^n(x)}\big)=0$ for all $x\in X$, 
	then $f$ has finite separation numbers.
\end{corollary}

For the second corollary, use Poincaré's Recurrence Theorem to get a contradiction.
It remains to prove part (a) of Theorem~\ref{t.automorphic_systems}, which we
restate as
\begin{theorem} \label{t.unbounded_separation}
	Let $f:X\to X$ be a continuous map.
	Further, assume that $f$ is a minimal almost 1-1 extension of an isometry
	$g:\Xi\to \Xi$ such that the factor map $h$ is not injective.
	Then $f$ has unbounded separation numbers.
\end{theorem}
For the proof, we will again need two preliminary lemmas.
Given $x,y\in X$ and $\delta>0$, we let 
\begin{equation}
	\label{e.separation_frequency}
	\nu(f,\delta,x,y) \ := \ \varlimsup_{n\to\infty} \ntel\countsep{n}(f,\delta,x,y) \ . 
\end{equation}
\begin{lemma} \label{l.point_separation}
	Suppose $V_1,V_2\ssq \Xi$ are two open sets which satisfy
	$d(h^{-1}(V_1),h^{-1}(V_2))\geq \delta$.
	Then $\nu(f,\delta,x_1,x_2)>0$ for all $x_1\in h^{-1}(V_1)$ and
	$x_2\in h^{-1}(V_2)$.
\end{lemma}
\begin{proof}
	Let $\xi_1:=h(x_1)$ and $\xi_2:=h(x_2)$.
	By assumption, we have that $d(f^k(x_1),f^k(x_2))\geq \delta$ whenever
	$g^k(\xi_1)\in V_1$ and $g^k(\xi_2)\in V_2$.
	Consequently,
	\begin{equation} \label{e.separation_frequency2}
		\nu(f,\delta,x_1,x_2) \ \geq \ \varlimsup_{n\to\infty} \ntel 
		\#\left\{ 0\leq k<n \mid (g\times g)^k(\xi_1,\xi_2)\in V_1\times V_2\right\} \ .
	\end{equation}
	However, as $g$ is an isometry, so is $g\times g$.
	This implies that all points $(\xi_1,\xi_2)\in \Xi\times \Xi$ are almost
	periodic, and the set of return times to any of their neighbourhoods is
	syndetic \cite{auslander1988minimal}.
	Hence, the right-hand side of \eqref{e.separation_frequency2} is strictly
	positive.
\end{proof}

\begin{lemma}\label{l.good_neighbourhoods}
	Suppose $f$ is a minimal almost 1-1 extension of $g$ and
	$\diam(h^{-1}(\xi))>\delta$ for some $\xi\in \Xi$.
	Then for every neighbourhood $U$ of $\xi$ there exist $V_1,V_2\ssq U$ such
	that $d(h^{-1}(V_1),h^{-1}(V_2)) > \delta$.
\end{lemma}
\begin{proof}
	Due to minimality, singleton fibres are dense in $X$.
	Hence, it is possible to find $x_1,x_2\in h^{-1}(U)$ such that $F_{x_i}=\{x_i\}$,
	$i\in\{1,2\}$ and $d(x_1,x_2)>\delta$.
	Then, by continuity, any sufficiently small neighbourhoods $V_i$ of $h(x_i)$
	will satisfy $d(h^{-1}(V_1),h^{-1}(V_2))>\delta$.
\end{proof}

\begin{proof}[\bf Proof of Theorem~\ref{t.unbounded_separation}.]
	Since the factor map $h$ is not injective, there exists $\xi\in \Xi$ with
	$\diam(h^{-1}(\xi))>\delta$ for some $\delta>0$.
	We will construct, by induction on $k\in\N$ with $k\geq 2$, a sequence of finite
	families of	disjoint open sets $V^k_1\ld V^k_k$ with the property that for
	all	$1\leq i<j\leq k$ there exists $n^k_{i,j}\in \N_0$ such that
	\begin{equation}\label{e.open_set_properties}
		d\left(h^{-1} \left( g^{n^k_{i,j}}\left(V^k_i\right)\right),
			h^{-1}\left(g^{n^k_{i,j}}\left(V^k_j\right)\right)\right) \ > \ \delta \ .
   \end{equation}
	For any family of points $x^k_i\in h^{-1}(V^k_i)$, $i\in\{1\ld k\}$, and
	$1\leq i<j\leq k$ we will then have 
	\[
		\nu\left(f,\delta,x^k_i,x^k_j\right) \ = \ 
		\nu\left(f,\delta,f^{n^k_{i,j}}\left(x^k_i\right),
			f^{n^k_{i,j}}\left(x^k_j\right)\right) \ > \ 0
	\]
	by Lemma~\ref{l.point_separation}.
	Thus, if $\nu_k:=\min\left\{\nu\left(f,\delta,x^k_i,x^k_j\right)\mid 1\leq i<j
		\leq k\right\}$, then $\{x^k_1\ld x^k_k\}$ is a
	$(f,\delta,\nu_k)$-separated set of cardinality $k$.
	This implies that $\sup_{\nu>0} \Sep(f,\delta,\nu)$
	is infinite, as required, since $k$ was arbitrary.

	It remains to construct the disjoint open sets $V^k_i$.
	For $k=2$, the sets $V^2_1$ and $V^2_2$ can be chosen according to
	Lemma~\ref{l.good_neighbourhoods} with $n^2_{1,2}=0$.
	Suppose that $V^k_1\ld V^k_k$ have been constructed as above.
	By minimality, there exists $n\in\N$ such that $g^n(V^k_k)$ is a
	neighbourhood of $\xi$.
	Lemma~\ref{l.good_neighbourhoods} yields the existence of open sets
	$V,V'\ssq g^n(V^k_k)$ with $d(h^{-1}(V),h^{-1}(V'))>\delta$.
	We now set 
	\[
		V^{k+1}_i\ := \ V^k_i\ \textnormal{ for }\ i\in\{1\ld k-1\}\ , \ 
		V^{k+1}_k \ := \ g^{-n}(V)
		\quad\textnormal{and}\quad V^{k+1}_{k+1}\ := \ g^{-n}(V') \ ,
	\]
	so that $V^{k+1}_k\cup V^{k+1}_{k+1}\ssq V^k_k$.
	Choosing $n^{k+1}_{i,j}:=n^k_{i,j}$ if $1\leq i<j\leq k-1$,
	$n^{k+1}_{i,j}:=n^k_{i,k}$ if $1\leq i\leq k-1$ and $j\in\{k,k+1\}$ and
	$n^{k+1}_{k,k+1}:=n$, we obtain that \eqref{e.open_set_properties} is
	satisfied for all $1\leq i<j\leq k+1$.
\end{proof}

\section{Properties of amorphic complexity and basic examples}
\label{sec_definitions}

\subsection{More general growth rates}

As mentioned in the introduction, one may consider more general than just
polynomial growth rates in the definition of amorphic complexity.
We call $a:\R_+\times(0,1]\to\R_+$ a {\em scale function} if $a(\,\cdot\,,\nu)$
is non-decreasing, $a(s,\,\cdot\,)$ is decreasing and
$\lim_{\nu\to 0}a(s,\nu)=\infty$ for all $s\in\R_+$.
If the separation numbers of $f$ are finite, then we let 
\begin{equation}\label{def_fsc_delta}
	\begin{aligned}
		\uac(f,a,\delta)&:=\sup\left\{s>0\;\left |\;\varliminf
			\limits_{\nu\to 0}\frac{\Sep(f,\delta,\nu)}{a(s,\nu)}>0\right.\right\} \ ,\\
		\oac(f,a,\delta)&:=\sup\left\{s>0\;\left|\;\varlimsup
			\limits_{\nu\to 0}\frac{\Sep(f,\delta,\nu)}{a(s,\nu)}>0\right.\right\}
	\end{aligned}
\end{equation}
and proceed to define the \emph{lower and upper amorphic complexity of $f$ with
respect to the scale function $a$} as
\begin{equation}
	\begin{aligned} \label{d.general_amorphic_compexity}
		\uac(f,a)&:=\sup\limits_{\delta>0}\,\uac(f,a,\delta) \ ,\\
		\oac(f,a)&:=\sup\limits_{\delta>0}\,\oac(f,a,\delta) \ .
	\end{aligned}
\end{equation}
As before, if $\uac(f,a)=\oac(f,a)$, then their common value is denoted by
$\fsc(f,a)$. If $a(s,\nu)=\nu^{-s}$, then this reduces to the definition given
in the introduction. 

In order to obtain good properties, however, some regularity has to be imposed
on the scale function.
We say a scale function $a$ is \emph{O-(weakly) regularly varying (at the origin)
with respect to $\nu$} if
\[
	\varlimsup\limits_{\nu\to 0}\frac{a(s,c\nu)}{a(s,\nu)}
\]
is finite for each $s,c>0$.
Under this assumption, a part of the theory can be developed in a completely
analogous way, until specific properties of polynomial growth start to play a role.
For the sake of simplicity, we refrain from stating the results in this section
in their full generality.
However, we provide extra comments in each subsection to specify the class of
scale functions the corresponding results extend to.
For more information on O-regularly varying functions, see for example
\cite{AljancicArandelovic1977,BuldyginKlesovSteinebach2006} and references therein.

\subsection{Definition via $(f,\delta,\nu)$-spanning sets}
\label{SpanningSets}

As in the case of topological entropy, amorphic complexity can be defined in an
equivalent way by using spanning sets instead of separating sets.
A subset $S$ of a metric space $(X,d)$ is said to be \emph{$(f,\delta,\nu)$-spanning}
if for all $x\in X$ there exists a $y\in S$ such that
\[
	\varlimsup\limits_{n\to\infty}\frac{\countsep{n}(f,\delta,x,y)}{n} \ < \ \nu \ .
\]
By $\Span(f,\delta,\nu)$ we denote the smallest cardinality of any
$(f,\delta,\nu)$-spanning set in $X$.

\begin{lemma}\label{lem_relations_sep_span}
	Let $f:X\to X$ be a map, $\delta>0$ and $\nu\in(0,1]$. We have that 
	\begin{equation}\label{e.span_sep_relation}
		\Sep(f,\delta,\nu)\ \geq \ \Span(f,\delta,\nu) \eqand
		\Span(f,\delta,\nu/2)\ \geq \ \Sep(f,2\delta,\nu)\ .
	\end{equation}
\end{lemma}
\begin{proof}
	For the first inequality, the proof is similar to the argument in the
	comparison of the separating and spanning sets in the classical definition
	of topological entropy \cite[Chapter 7.2]{Walters1982}.

	For the second inequality, assume \twlog\ that $\Span(f,\delta,\nu/2)<\infty$.
	Let $S\subseteq X$ be an $(f,\delta,\nu/2)$-spanning set of cardinality
	$\Span(f,\delta,\nu/2)$ and assume for a contradiction that $\tilde S\subseteq X$
	is an $(f,2\delta,\nu)$-separated set with $\#\tilde S>\#S$.
	Then for some $y\in S$ there exist $x_1,x_2\in\tilde S$ such that
	\[
		\varlimsup\limits_{n\to\infty}\frac{\countsep{n}(f,\delta,x_i,y)}{n}
			\ < \ \frac{\nu}{2}
	\]
	with $i\in\{1,2\}$. However, due to the triangle inequality we have that
	\[
		\countsep{n}(f,2\delta,x_1,x_2) \ \leq \
		\countsep{n}(f,\delta,x_1,y)+\countsep{n}(f,\delta,y,x_2) \ 
	\]
	and consequently
	\[
        \varlimsup\limits_{n\to\infty}\frac{\countsep{n}(f,2\delta,x_1,x_2)}{n} \
        \leq \
        \varlimsup\limits_{n\to\infty}\frac{\countsep{n}(f,\delta,x_1,y)}{n} +
        \varlimsup\limits_{n\to\infty}\frac{\countsep{n}(f,\delta,x_2,y)}{n}
        \ < \ \nu \ .
	\]
	This contradicts the fact that $x_1$ and $x_2$ are $(f,2\delta,\nu)$-separated. 
\end{proof}
\begin{corollary}\label{c.span_sep}
	Given a metric space $X$ and $f:X\to X$, we have that 
	\begin{equation}
		\uac(f) \ = \ \sup_{\delta>0} \varliminf_{\nu\to 0}
			\frac{\log\Span(f,\delta,\nu)}{-\log\nu} 
		\eqand
		\oac(f) \ = \ \sup_{\delta>0} \varlimsup_{\nu\to 0}
			\frac{\log\Span(f,\delta,\nu)}{-\log\nu} \ . 
	\end{equation}
\end{corollary}
\begin{remarks}\mbox{}
	\alphlist
		\item The above statement remains true if $a(s,\nu)=\nu^{-s}$ is
			replaced by any O-regularly varying scale function.
		\item In the definition of $(f,\delta,\nu)$-separated sets and
			$(f,\delta,\nu)$-spanning sets one could also use $\liminf$ instead
			of $\limsup$, and thus define the notions of
			\emph{strongly $(f,\delta,\nu)$-separated sets} and
			\emph{weakly $(f,\delta,\nu)$-spanning sets}, respectively.
			However, there is no analogue to the second inequality in
			\eqref{e.span_sep_relation} in this case.
	\listend
\end{remarks}

\subsection{Factor relation and topological invariance}

We assume that $X$ and $\Xi$ are arbitrary metric spaces, possibly non-compact.
The price to pay for this is that we have to assume the uniform continuity of
the factor map.
All the assertions of this section remain true for arbitrary scale functions.
\begin{proposition}
	Assume $g:\Xi\to \Xi$ is a factor of $f:X\to X$ with a uniformly continuous
	factor map $h:X\to \Xi$.  Then $\uac(f)\geq\uac(g)$ and $\oac(f)\geq
	\oac(g)$.
\end{proposition}
\begin{proof}
	We denote the metric on $X$ and $\Xi$ with $d$ and $\rho$, respectively.
	The uniform continuity of $h$ implies that for every $\delta>0$ there exists
	$\tilde\delta>0$ such that $\rho(h(z),h(w))\geq\delta$ implies
	$d(z,w)\geq\tilde\delta$.
	Suppose $\xi,\xi'\in \Xi$ are $(g,\delta,\nu)$-separated.
	Then there exist $x,x'\in X$ such that $h(x)=\xi$ and $h(x')=\xi'$.
	Since $\rho(g^k(\xi),g^k(\xi'))\geq\delta$ implies
	$d(f^k(x),f^k(x'))\geq\tilde\delta$, the points $x$ and $x'$ need to be 
	$(f,\tilde \delta,\nu)$-separated.
	Given $\nu\in(0,1]$, this means that if $S\subseteq \Xi$ is a 
	$(g,\delta,\nu)$-separated set, then there exist $\tilde
	S\subseteq X$ with $h(\tilde S)=S$ and $\tilde\delta>0$ such that $\tilde S$
	is a $(f,\tilde\delta, \nu)$-separated set.
	Therefore, for all $\nu\in(0,1]$ we get
	\begin{align*}
		\Sep(f,\tilde\delta,\nu)\geq \Sep(g,\delta,\nu) \ .
	\end{align*}
	The assertions follow easily.		
\end{proof}

\begin{corollary}\label{fsc_top_inv}
	Suppose $X$ and $\Xi$ are compact and let $f:X\to X$ and $g:\Xi\to \Xi$
	be conjugate.
	Then $\uac(f)=\uac(g)$ and $\oac(f)=\oac(g)$.
\end{corollary}

For the next corollary, observe that $f\circ g$ is an extension of $g\circ f$
with factor map $h=g$, and conversely $\tilde h=f$ is a factor map from $g\circ
f$ to $f\circ g$.

\begin{corollary}
	Suppose $f:X\to X$ and $g:X\to X$ are uniformly continuous.
	Then $\uac(f\circ g)=\uac(g\circ f)$ and $\oac(f\circ g)=\oac(g\circ f)$.
\end{corollary}

\subsection{Power invariance and product rule}

We first consider iterates of $f$. In contrast to topological entropy, taking
powers does not affect the amorphic complexity.
Throughout this section, we assume that $X$ and $Y$ are metric spaces. 
\begin{proposition}
	Assume $f:X\to X$ is uniformly continuous and let $m\in\N$.
	Then $\uac(f^m)=\uac(f)$ and $\oac(f^m)=\oac(f)$.
\end{proposition}
\begin{proof}
	Since all iterates of $f$ are uniformly continuous as well, we have that for
	every $\delta>0$ there exists $\tilde\delta>0$ such	that
	$d(f^i(z),f^i(w))\geq\delta$ implies $d(z,w)\geq\tilde\delta$ for
	all $i\in\{0\ld m-1\}$.

	Suppose $x,y\in X$ are $(f,\delta,\nu)$-separated.
	Assume that $d(f^k(x),f^k(y))\geq\delta$ with $k=m\cdot\tilde k+i$, where
	$\tilde k\in\N_0$ and $i\in\{0,\dots,m-1\}$. 
    Then by the above we have
    $d\big(f^{m\tilde k}(x),f^{m\tilde k}(y)\big)\geq\tilde\delta$.
	This means that for $\tilde n\in\N$ and
	$n\in\{m\cdot\tilde n,\dots,m(\tilde n+1)-1\}$ we get
	\[
		\ntel\countsep{n}(f,\delta,x,y) \ \leq \ 
		\ntel\left(m\cdot\countsep{\tilde n}(f^m,\tilde\delta,x,y)+m\right)
		\ \leq \ \frac{1}{\tilde n}(\countsep{\tilde n}(f^m,\tilde\delta,x,y)+1)\ . 
	\]
	By taking the $\limsup$ we get that $x$ and $y$ are
	$(f^m,\tilde\delta,\nu)$-separated.
	Hence,
	\begin{align}\label{prop_iterates_of_f_1}
		\Sep(f^m,\tilde\delta,\nu)\ \geq \ \Sep(f,\delta,\nu) \ .
	\end{align}

	Conversely, suppose that $x$ and $y$ are $(f^m,\delta,\nu)$-separated. 
	Then for $k\geq 1$ it follows from $d(f^{mk}(x),f^{mk}(y))\geq\delta$ that
	$d\big(f^{\tilde k}(x),f^{\tilde k}(y)\big)\geq\tilde\delta$ for all
	$\tilde k\in\{m(k-1)+1,\dots,mk\}$.
	Each $\tilde n\in\N$ belongs to a block $\{m(n-1)+1,\dots,m\cdot n\}$ with
	$n\in\N$ and we have
	\[
		\frac{1}{\tilde n}\countsep{\tilde n}(f,\tilde\delta,x,y)  \ \geq \ 
		\frac{1}{\tilde n}\left(m\cdot\countsep{n}(f^m,\delta,x,y)-m\right)
		\ \geq\ \ntel\left(\countsep{n}(f^m,\delta,x,y)-1\right) \ .
	\]
	Again, by taking the $\limsup$ we get that $x$ and $y$ are
    $(f,\tilde\delta,\nu)$-separated.
    Hence,
	\begin{align}\label{prop_iterates_of_f_2}
		\Sep(f,\tilde\delta,\nu) \ \geq \ \Sep(f^m,\delta,\nu) \ .
	\end{align}
	Using \eqref{prop_iterates_of_f_1} and \eqref{prop_iterates_of_f_2}, we get
	that $\uac(f^m)=\uac(f)$ and $\oac(f^m)=\oac(f)$.
\end{proof}

\begin{remarks}\mbox{}
	\alphlist
		\item The above result remains true for arbitrary scale functions.
		\item If $f$ is not uniformly continuous, then we still have
			$\Sep(f,\delta,\nu/m)\geq\Sep(f^m,\delta,\nu)$.
			This yields $\uac(f,a)\geq\uac(f^m,a)$ and $\oac(f,a)\geq\oac(f^m,a)$
			for $a$ O-regularly varying.
	\listend
\end{remarks}

In contrast to the above, the product formula is specific to polynomial growth
or, more generally, to scale functions satisfying a product rule of the form
$a(s+t,\nu)=a(s,\nu)\cdot a(t,\nu)$.
\begin{proposition}
	Let $f:X\to X$ and $g: Y\to Y$.
	Then $\uac(f\times g)\geq\uac(f)+\uac(g)$
	and $\oac(f\times g)\leq \oac(f)+\oac(g)$.
	Therefore, if the limits $\fsc(f)$ and $\fsc(g)$ exist,	we get
	\[
		\fsc(f\times g)=\fsc(f)+\fsc(g) \ .
	\]
\end{proposition}
\begin{proof}
	We denote the metric on $X$ and $Y$ by $d_X$ and $d_Y$, respectively.
	Let $d$ be the maximum metric on the product space $X\times Y$.
	Using Corollary~\ref{c.span_sep}, the assertions are direct consequences of
	the following two inequalities, which we show for all $\delta>0$ and
	$\nu\in(0,1]$ 
	\begin{eqnarray}
		\label{sep_product_inequality}\Sep(f\times g,\delta,\nu) & \geq &
		\Sep(f,\delta,\nu)\cdot \Sep(g,\delta,\nu) \ , \\
		\label{span_product_inequality}\Span(f\times g,\delta,\nu) & \leq &
			\Span(f,\delta,\nu/2)\cdot\Span(g,\delta,\nu/2) \ .
	\end{eqnarray}
	For proving \eqref{sep_product_inequality} assume that $S_X\subseteq X$
	and $S_Y\subseteq Y$ are $(f,\delta,\nu)$- and $(g,\delta,\nu)$-separated
	sets, respectively, with cardinalities $\Sep(f,\delta,\nu)$ and
	$\Sep(g,\delta,\nu)$, respectively.
	Then $S:=S_X\times S_Y\subseteq X\times Y$ is an
	$(f\times g,\delta,\nu)$-separated set.
	This implies \eqref{sep_product_inequality}.
	
	Now, in order to prove \eqref{span_product_inequality} assume that
	$\tilde S_X\subseteq X$ and $\tilde S_Y\subseteq Y$ are	$(f,\delta,\nu/2)$-
	and $(g,\delta,\nu/2)$-spanning sets, respectively,	with cardinalities
	$\Span(f,\delta,\nu/2)$ and $\Span(g,\delta,\nu/2)$, respectively.
	The set $\tilde S:=\tilde S_X\times\tilde S_Y\subseteq X\times Y$ is
	$(f\times g,\delta,\nu)$-spanning, since for arbitrary $(x,y)\in X\times Y$
	there are $\tilde x\in\tilde S_X$ and $\tilde y\in\tilde S_Y$ such that
	\begin{multline*}
		\countsep{n}(f\times g,\delta,(x,y),(\tilde x,\tilde y))
		=\#\left\{0\leq k<n\;|\;d\big((f\times g)^k(x,y),
			(f\times g)^k(\tilde x,\tilde y)\big)\geq\delta\right\}\\
		\leq\#\left\{0\leq k<n\;|\; d_X(f^k(x),f^k(\tilde x))\geq\delta\right\}
			+\#\left\{0\leq k<n\;|\; d_Y(g^k(y),g^k(\tilde y))\geq\delta\right\}
		\ .\qedhere
	\end{multline*}
\end{proof}

\subsection{Isometries, Morse-Smale systems and transient dynamics}\label{Isometries}

It is obvious that all isometries have bounded separation numbers and zero
amorphic complexity, since $\Sep(f,\delta,\nu)$ does not depend on $\nu$ in this
case.
Similarly, amorphic complexity is zero for Morse-Smale systems.
Here, we call a continuous map $f$ on a compact metric space $X$
\emph{Morse-Smale} if its non-wandering set $\Omega(f)$ is finite.
This implies that $\Omega(f)$ consists of a finite number of fixed or periodic
orbits, and for any $x\in X$ there exists $y\in\Omega(f)$ with
$\nLim f^{np}(x)=y$, where $p$ is the period of $y$.
Since orbits converging to the same periodic orbit cannot be
$(f,\delta,\nu)$-separated, we obtain $\Sep(f,\delta,\nu)\leq \#\Omega(f)$ for
all $\delta,\nu>0$.
Hence, separation numbers are even bounded uniformly in $\delta$ and $\nu$.

This shows that amorphic complexity is, in some sense, less sensitive to transient
behaviour than power entropy, which gives positive value to Morse-Smale systems
(see Section~\ref{PowerEntropy}).
However, amorphic complexity is not entirely insensitive to transient dynamics,
and the relation $\fsc(f)=\fsc(\left.f\right|_{\Omega(f)})$ does not always hold.
An example can be given as follows.

Let $f:[0,1]\times\kreis\to[0,1]\times \kreis$ be of the form
$f(x,y):=(g(x),y+\alpha(x)\bmod 1)$, where $\T^1:=\R^1/\Z^1$, $\alpha:[0,1]\to\R$
is continuous and $g:[0,1]\to[0,1]$ is a Morse-Smale homeomorphism with unique
attracting fixed point $x_a=0$ and unique repelling fixed point $x_r=1$, so that
$\kLim g^k(x)=0$ for all $x\in(0,1)$.
Let $x_0\in(0,1)$ and $x_k:=g^k(x_0)$ for $k\in\N$ and $x_0':=(x_0+x_1)/2$.
Suppose $\alpha$ is given by 
\begin{equation}\label{e.alpha}
	\alpha(x) \ := \ \left\{
	\begin{array}{cl}
		0 & \textrm{if } x\in\{0\}\cup (x_0,1] ;\\
		1-2\frac{|x_0'-x|}{x_0-x_1} & \textrm{if } x\in (x_1,x_0];\\ 
		\ktel \alpha\left(g^{-(k-1)}(x)\right) & \textrm{if } x
			\in (x_k,x_{k-1}],\ k\geq 2;   
	\end{array}\right. \ .
\end{equation}
Then, if $x,x'\in [x_1,x_0']$, we have that 
\begin{equation}\label{e.different_speeds}
	\left| \sum_{k=0}^{n-1} \alpha\circ g^k(x) - \sum_{k=0}^{n-1} \alpha\circ
		g^k(x')\right| 
	\ = \ 2\frac{\abs{x-x'}}{x_0-x_1}\sum _{k=1}^{n} \frac{1}{k} \ .
\end{equation}
This means that one of the two points $(x,0),(x',0)$ performs infinitely more
turns around the annulus $[0,1]\times\kreis$ as $n\to\infty$, and it is not
difficult to deduce from \eqref{e.different_speeds} that $(x,0),(x',0)$ are
$(f, \delta,\nu)$-separated for some fixed $\delta,\nu>0$ independent of $x,x'$.
Hence, $[x_1,x_0']\times\{0\}$ is an uncountable $(f,\delta,\nu)$-separated set,
and we obtain $\Sep(f,\delta,\nu)=\infty$. 

It should be interesting to describe which types of transient behaviour have
an impact on amorphic complexity and which ones do not, and thus to understand
whether this quantity may be used to distinguish qualitatively different types
of transient dynamics.
However, we are not going to pursue this issue further here, but confine
ourselves to give a simple criterion for the validity of the equality
$\fsc(f)=\fsc(\left.f\right|_{\Omega(f)})$.

We say $f$ has the \emph{unique target property} if for every $x\in X\smin\Omega(f)$
there exists $y\in\Omega(f)$ such that $\nLim d(f^n(x),f^n(y))=0$.
Then the following statement is easy to prove.
\begin{lemma} \label{l.unique_target}
	Assume $f$ has the unique target property, then  
	$\uac(f)=\uac\big(\left.f\right|_{\Omega(f)}\big)$
	and	$\oac(f)=\oac\big(\left.f\right|_{\Omega(f)}\big)$.
\end{lemma}
In fact, the nonwandering set $\Omega(f)$ does not play a special role in the
definition of the unique target property nor in the above lemma and can be
replaced by any other subset of $X$ (even invariance is not necessary).
For later use (see Section~\ref{PinchedSystems}), we provide a precise formulation.
Given $E\ssq X$, we let
\begin{equation}\label{e.subset_separation_numbers}
  \Sep_{E}(f,\delta,\nu) \ 
   := \ \sup\left\{\#A \mid A\ssq E \textrm{ and } A \textrm{ is }(f,\delta,\nu)
  \textrm{-separated}\right\} 
\end{equation}
and define
\begin{align}\label{eq: defn h am for subsets}
	\begin{split}
	 	\uac_{E}(f,\delta) & \ := \ \varliminf_{\nu\to 0}
			\frac{\log \Sep_{E}(f,\delta,\nu)}{-\log \nu}\quad , \quad
		\uac_{E}(f) \ := \ \sup_{\delta>0}\uac_{E}(f,\delta)\ ,\\
		\oac_{E}(f,\delta) & \ := \ \varlimsup_{\nu\to 0}
			\frac{\log\Sep_{E}(f,\delta,\nu)}{-\log \nu}\quad , \quad
		\oac_{E}(f)\, \ := \ \sup_{\delta>0}\oac_{E}(f,\delta) \ .
	\end{split}
\end{align}

We say $f$ has the \emph{unique target property with respect to $E\ssq X$} if
for all $x\in X$ there exists $y\in E$ such that $\nLim d(f^n(x),f^n(y))=0$.
\begin{lemma}\label{lem: unique target general}
	Suppose $f:X\to X$ has the unique target property with respect to $E\ssq X$.
	Then $\uac(f)=\uac_{E}(f)$ and $\oac(f)=\oac_{E}(f)$.
\end{lemma}
\begin{proof}
	Suppose $S=\{x_1\ld x_m\} \ssq X$ is an $(f,\delta,\nu)$-separated set.
	Then by assumption there exist $y_1\ld y_m\in E$ such that
	$\nLim d(f^n(x_i),f^n(y_i))=0$ for all $i\in\{1\ld m\}$.
	Hence, the set $\tilde S:=\{y_1\ld y_m\}\ssq E$ is $(f,\delta,\nu)$-separated
	as well.
	This shows that $\Sep_{E}(f,\delta,\nu)\geq\Sep(f,\delta,\nu)$, and since
	the reverse inequality is obvious this proves the statement.
\end{proof}

\subsection{Denjoy examples and Sturmian subshifts}
\label{BasicExamples}

We start with some standard notation concerning circle maps and symbolic dynamics.
Let $\T^1=\R/\Z$ be the circle and denote by $d$ the usual metric on $\T^1$.
Further, we denote the open and the closed counter-clockwise interval from $a$
to $b$ in $\T^1$ by $(a,b)$ and $[a,b]$, respectively.
The Lebesgue measure on $\T^1$ is denoted by $\Leb$.
Moreover, the rigid rotation with angle $\alpha\in\R$ is denoted by
$R_\alpha(x):=x+\alpha\mod 1$.

For a finite set $A$ we denote by $\sigma$ the left shift on $\Sigma_A:=A^{\I}$
where $\I$ equals either $\N_0$ or $\Z$.
The product topology on $\Sigma_A$ is induced by the \emph{Cantor metric}
$\rho(x,y):=2^{-j}$ where $x=(x_k)_{k\in\I}$, $y=(y_k)_{k\in\I}\in\Sigma_A$
and $j:=\min\{|k|: x_k\neq y_k\textnormal{ with } k\in\I\}$.\smallskip

We first recall some basics about Sturmian subshifts and Denjoy homeomorphisms
of the circle.
For Sturmians, we mainly follow \cite[Section 2.2]{ChazottesDurand2005}.
Assume that $\alpha\in (0,1)$ is irrational.
Consider the coding map $\phi_{\alpha}:\T^1\to\{0,1\}$ defined via $\phi_{\alpha}(x)=0$
if $x\in I_0:=[0,1-\alpha)$ and $\phi_{\alpha}(x)=1$ if $x\in I_1:=[1-\alpha,1)$.
Set
\[
	\Sigma_{\alpha}\ := \ \overline{
	\left\{(\phi_{\alpha}(R_{\alpha}^k(x)))_{k\in\Z}\;|\;x\in\T^1\right\}}
	\ \subset \ \Sigma_{\{0,1\}} \ .
\]
The subshift $(\Sigma_{\alpha},\sigma)$ is called the
\emph{Sturmian subshift generated by $\alpha$} and its elements are called
\emph{Sturmian sequences}.
According to \cite{MorseHedlund1940}, there exists a map $h:\Sigma_{\alpha}\to\T^1$
semi-conjugating $\sigma$ and $R_{\alpha}$ with the property that $\#h^{-1}(x)=2$
for $x\in\{k\alpha\mod 1\;|\;k\in\Z\}$ and $\#h^{-1}(x)=1$ otherwise.
If $x=k\alpha$, then one of the two alternative sequences in $h^{-1}(x)$
corresponds to the coding with respect to the original partition $\{I_0, I_1\}$,
whereas the other one corresponds to the coding with respect to the partition
$\{(0,1-\alpha],(1-\alpha,1]\}$.
Further information is given in \cite[Section 1.6]{BlanchardMaassNogueira2000}.\smallskip

Poincaré's classification of circle homeomorphisms in \cite{Poincare1885} states
that to each orientation preserving homeomorphism $f:\T^1\to\T^1$ of the circle
we can associate a unique real number $\alpha\in[0,1)$, called the rotation
number of $f$, such that $f$ is semi-conjugate, via an orientation preserving
map, to the rigid rotation $R_{\alpha}$, provided $\alpha$ is irrational (see
also \cite{demelo/vanstrien:1993,HasselblattKatok1997}).
Another classical result by A. Denjoy \cite{Denjoy1932} states that if $f$ is a
diffeomorphism such that its derivative is of bounded variation, then $f$ is
even conjugate to $R_{\alpha}$.
In this case, the amorphic complexity is zero.
However, Denjoy also constructed examples of $C^1$ circle diffeomorphism with
irrational rotation number that are not conjugate to a rotation and later, Herman
\cite{Herman1979} showed that these examples can be made $C^{1+\varepsilon}$ for
any $\varepsilon<1$.
Such maps are commonly called \emph{Denjoy examples} or {\em Denjoy homeomorphisms}.
From Poincar\'e's classification, it is known that in this case there exist
\emph{wandering intervals}, that is, open intervals $I\subset \kreis$ such that
$f^n(I)\cap I=\emptyset$ for all $n\geq 1$.
Any Denjoy example has a unique minimal set $C$, which is a Cantor set and
coincides with the non-wandering set $\Omega(f)$.
All connected components of $\kreis\smin C$ are wandering intervals, and the
length of their $n$-th iterates goes to zero as $n\to\infty$.
Since the endpoints of these intervals belong to the minimal set, this also
implies that Denjoy examples have the unique target property. \smallskip

Not surprisingly, there is an intimate connection between Denjoy examples and
Sturmian subshifts.
Let $f$ be a Denjoy homeomorphism with rotation number $\alpha$ and suppose it
has a unique wandering interval $I$, in the sense that the minimal set
$C=\kreis\smin\bigcup_{n\in\Z} f^n(I)$.\foot{It is possible
	to have several connected components of $\kreis\smin C$ with pairwise
	disjoint orbits.}
Given any $x_0\in I$, let $J_0:=[f(x_0),x_0)$ and $J_1:=[x_0,f(x_0))$.
Then for every $x\in\kreis$ the coding $\ind_{J_1}\circ f^n(x)$, where $\ind_{J_1}$
denotes the indicator function of $J_1$, is a Sturmian sequence in $\Sigma_\alpha$.
Moreover, in this situation any point in the minimal set $C$ has a unique coding.
This yields the following folklore statement.
\begin{lemma} \label{l.Sturmian_Denjoy}
	For any Sturmian subshift $(\Sigma_\alpha,\sigma)$ there exists a Denjoy
	homeomorpism $f$ with minimal set $C$ such that $\left.f\right|_{C}$
	is conjugate to $\left.\sigma\right|_{\Sigma_\alpha}$.
\end{lemma}
For our purposes, this means that we only have to determine the amorphic
complexity of Denjoy examples.
Note that the converse to the above lemma is false: if $f$ has multiple
wandering intervals with pairwise disjoint orbits, then it is not conjugate to a
Sturmian subshift.
\begin{theorem} \label{t.denjoy}
	Suppose $f:\T^1\to\T^1$ is a Denjoy homeomorphism.
	Then $\fsc(f)=1$.
\end{theorem}
Since Denjoy examples have the unique target property,
Lemma~\ref{l.unique_target} yields that $\fsc(\left.f\right|_{C})=\fsc(f)=1$.
Together with Corollary~\ref{fsc_top_inv} and Lemma~\ref{l.Sturmian_Denjoy},
this implies
\begin{corollary}
	For any Sturmian subshift $(\Sigma_\alpha,\sigma)$ we have
	$\fsc\big(\left.\sigma\right|_{\Sigma_\alpha}\big)=1$.
\end{corollary}

Theorem~\ref{t.denjoy} is a direct consequence of the following two lemmas.
However, before we proceed, we want to collect some more facts concerning
Denjoy examples, following mainly \cite[Section 0]{Markley1970} and
\cite[Section 2]{Hernandez-CorbatoOrtega2012}.
The Cantor set $C=\Omega(f)$ can be described as
\[
	C\ = \ \T^1\backslash\bigcup\limits_{\ell=1}^\infty (a_{\ell},b_{\ell}) \ ,
\]
where $((a_{\ell},b_{\ell}))_{\ell\in\N}$ is a family of open and pairwise
disjoint intervals.
The \emph{accessible points $A\subset\T^1$ of $C$} are defined as the union of
the endpoints of these intervals and the \emph{inaccessible points of $C$} are
defined as $I:=C\backslash A$.
A \emph{Cantor function $p:\T^1\to\T^1$ associated to $C$} is a continuous map
satisfying
\[
	p(x)=p(y)\ \Longleftrightarrow \ x=y\textnormal{ or } x,y
	\in [a_{\ell},b_{\ell}] \textnormal{ for some }\ell\geq 1 \ ,
\]
that is, $p$ collapses the intervals $[a_{\ell},b_{\ell}]$ to single points and is
invertible on $I$.
From this definition it is not difficult to deduce that $p$ is onto and that
$p(A)$ is countable and dense in $\T^1$.
Furthermore, we can assume without loss of generality that
$p\circ f=R_{\alpha}\circ p$, where $\alpha\in [0,1)\backslash\Q$ is the rotation
number of $f$, see \cite[Section 2]{Markley1970}.
\begin{lemma}
	Let $f:\T^1\to\T^1$ be a Denjoy homeomorphism.
	Then there exists $\delta>0$ such that
	$\Sep(f,\delta,\nu)\geq\lfloor 1/\nu\rfloor$ for all $\nu\in(0,1]$.
\end{lemma}
Note that by definition this implies that $\uac(f)\geq 1$.
\begin{proof}
	Suppose $\nu\in(0,1/2]$.
	Since $p(A)$ is dense in $\T^1$, we can choose for each $m\in\{1,2,3\}$ a
	point $\zeta_m\in p(A)$ such that
	\begin{equation}\label{property_zeta_m_s} 
		d(\zeta_m,\zeta_n) \ > \ 1/4 \quad \textrm{for } m\neq n \ .
	\end{equation}
	Note that to each $\zeta_m$ we can associate an interval
	$[a_{\ell_m},b_{\ell_m}]$ with $p\big([a_{\ell_m},b_{\ell_m}]\big)=\{\zeta_m\}$.
	Now, choose $\delta>0$ such that
	\[
		\delta\ \leq \ \min\limits_{m=1}^3d\big(a_{\ell_m},b_{\ell_m}\big) \ .
	\]
        Since $p(I)$ has full Lebesgue measure in $\T^1$, we can choose a set of
        $\lfloor 1/\nu\rfloor$ points
	\[
		M\ =\ \big\{x_1,\dots,x_{\lfloor 1/\nu\rfloor}\big\}\ \subset\  I \ ,
	\]
	such that $p(M)$ is an equidistributed lattice in $\T^1$ with distance
	$1/\lfloor 1/\nu\rfloor \geq \nu$ between adjacent vertices.
	Consider distinct points $x_i, x_j\in M$ and assume without loss of
	generality that $\Leb([p(x_i),p(x_j)])\leq 1/2$.
	Set $P:=[p(x_i),p(x_j)]$.
	If $\zeta_1\in R_{\alpha}^k(P)$ for some $k\geq 0$, then due to
	\eqref{property_zeta_m_s} we have that $\zeta_2\in\T^1\backslash
	R_{\alpha}^k(P)$ or $\zeta_3\in\T^1\backslash R_{\alpha}^k(P)$, such
	that both $[f^k(x_i),f^k(x_j)]$ and $[f^k(x_j),f^k(x_i)]$ contain some
	interval $[a_{\ell_m},b_{\ell_m}]$ with $m\in\{1,2,3\}$.
	Hence, we have
	\[
		d(f^k(x_i),f^k(x_j))\ \geq\ \delta \ .
	\]
	Consequently, we obtain
	\[
		\frac{\countsep{n}(f,\delta,x_i,x_j)}{n}\ \geq \ 
		\frac{\#\left\{0\leq k<n\;|\;\zeta_1\in R_{\alpha}^k(P)\right\}}{n}  \ .
	\]
	By Weyl’s Equidistribution Theorem \cite[Example 4.18]{EinsiedlerWard2011},
	the right-hand side converges to $p(x_j)-p(x_i)\geq\nu$ as $n\to\infty$.
	This means that $x_i$ and $x_j$ are $(f,\delta,\nu)$-separated, so that $M$
	is an $(f,\delta,\nu)$-separated set.
\end{proof}

\begin{lemma}
	Let $f:\T^1\to\T^1$ be a Denjoy homeomorphism.
	Then for any $\delta>0$ there exists a constant $\kappa=\kappa(\delta)$ such
	that 
	\[
        \Span(f,\delta,\nu) \ \leq \ \kappa/\nu \quad \textrm{for all }
        \nu\in(0,1] \ .
	\]
\end{lemma}
Together with Corollary \ref{c.span_sep}, this implies that $\oac(f)\leq 1$,
thus completing the proof of Theorem~\ref{t.denjoy}.
\begin{proof}
	We show that if $0<\tilde\nu\leq 1/(2(\lceil 1/\delta\rceil+1))$, then
	\[
		\Span(f,\delta,2\tilde\nu(\lceil 1/\delta\rceil+1))
		\ \leq \ \lceil 1/\tilde\nu\rceil \ .
	\]
    Since $\lceil 1/\tilde \nu\rceil \leq 2/\tilde \nu$, this implies the
	statement with $\kappa(\delta):=4(\lceil 1/\delta\rceil+1)$. 

	Let $\mu:=\Leb\circ p^{-1}$ and define the function
	$\varphi_{\tilde\nu}:\T^1\to[0,\infty)$ by
	\[
		\varphi_{\tilde\nu}(x):=\mu([x,x+\tilde\nu]) \ .
	\]
	Note that $d(x,y)\leq\mu([p(x),p(y)])$ and that
	$\varphi_{\tilde\nu}(x)=d(p^{-1}(x),p^{-1}(x+\tilde\nu))$ almost everywhere.
	In particular, $\varphi_{\tilde\nu}$ is measurable.
	Now, consider a subset $\tilde I\subseteq I$ such that
	\begin{align}\label{property_ba_varphi_alpha}
		\frac{\#\left\{0\leq k<n\;|\;
                    \varphi_{\tilde\nu}(R_{\alpha}^k(x))\geq\delta\right\}}{n} \
                \longrightarrow\ \Leb(\{x\in\kreis\mid
                \varphi_{\tilde\nu}(x)\geq\delta\})
        \quad\textrm{as}\quad n\to\infty 
	\end{align}
	for all $x\in p(\tilde I)$.
	Let $\{\varphi_{\tilde\nu}\geq\delta\}:=\{x\in\kreis\mid
		\varphi_{\tilde\nu}(x)\geq\delta\}$.
	Using Birkhoff's Ergodic Theorem, we know that $\tilde I$ can be chosen such
	that $p(\tilde I)$ has full Lebesgue measure.
	Hence, we can choose a set of $\lceil 1/\tilde\nu\rceil$ points
	\[
		M:=\big\{x_1,\dots,x_{\lceil 1/\tilde\nu\rceil}\big\}\subset\tilde I,
	\]
	such that $p(M)$ is an equidistributed lattice in $\T^1$ with distance
	$1/\lceil 1/\tilde\nu\rceil \leq \tilde\nu$ between adjacent
	vertices.
	Our aim is to show that $M$ is an
	$(f,\delta,2\tilde\nu(\lceil 1/\delta\rceil+1))$-spanning set. 

	For arbitrary $y\in\T^1$, let $x_i,x_j\in M$ be the two adjacent lattice
	points with $p(y)\in [p(x_i),p(x_j)]$ (that is, $j=i+1$ or $i=\lceil
	1/\tilde\nu\rceil$ and $j=1$).
	Then
	\[
		R_{\alpha}^k[p(x_i),p(y)]\ 
		\subseteq\ [R_{\alpha}^k(p(x_i)),R_{\alpha}^k(p(x_i))+\tilde\nu]
	\]
	for $k\geq 0$, and this implies
	\begin{eqnarray*}
		\lefteqn{	d\left(f^k(x_i),f^k(y)\right)
		 \ \leq \ \mu\left([p(f^k(x_i)),p(f^k(y))]\right)} \\
		& = & \mu\left(R_{\alpha}^k[p(x_i),p(y)]\right)
		\ \leq \ \varphi_{\tilde\nu}(R_{\alpha}^k(p(x_i))) \ .
	\end{eqnarray*}  
	We get that
	\[
		\frac{\countsep{n}(f,\delta,x_i,y)}{n}\ \leq \ 
		\frac{\#\left\{0\leq k<n\;|\;
			\varphi_{\tilde\nu}(R_{\alpha}^k(p(x_i)))\geq\delta\right\}}{n}
	\]
	and using \eqref{property_ba_varphi_alpha} we know that the right-hand side
	convergences to $\Leb(\{\varphi_{\tilde\nu}\geq\delta\})$ as $n\to\infty$.
	
	It remains to show that $\Leb(\{\varphi_{\tilde\nu}\geq\delta\})<
		2\tilde\nu(\lceil 1/\delta\rceil+1)$.
	Suppose for a contradiction that this inequality does not hold.
	Then $\{\varphi_{\tilde\nu}\geq \delta\}$ is not contained in a union of
	less than $\lceil 1/\delta\rceil+1$ intervals of length $2\tilde\nu$.
	Consequently, there exist at least $\lceil 1/\delta\rceil+1$ points
	$\zeta_i\in\T^1$ with $\varphi_{\tilde\nu}(\zeta_i)\geq\delta$ and
	$d(\zeta_i,\zeta_j)\geq\tilde\nu$ for $i\neq j$. We thus obtain
	\[
		\mu(\T^1)\ \geq \ \sum\limits_{i=1}^{\lceil 1/\delta\rceil+1}
		\mu([\zeta_i,\zeta_i+\tilde\nu])
		\ = \ \sum\limits_{i=1}^{\lceil 1/\delta\rceil+1}
		\varphi_{\tilde\nu}(\zeta_i)\ \geq \  1+\delta\ > \ 1 \ ,
	\]
	which is a contradiction.
	
	This means $\varlimsup_{n\to\infty}\countsep{n}(f,\delta,x_i,y)/n
		\leq\Leb(\{\varphi_{\tilde\nu}\geq\delta\})
		<2\tilde\nu(\lceil1/\delta\rceil+1)$,
	and since $y$ was arbitrary this shows that $M$ is
	an $(f,\delta,2\tilde\nu(\lceil 1/\delta\rceil+1))$-spanning set.
	This completes the proof.
\end{proof}

\subsection{Relations to power entropy}\label{PowerEntropy}

Given a compact metric space $(X,d)$ and a continuous map $f:X\to X$, the {\em
Bowen-Dinaburg metrics} are given by $d_n(x,y):=\max_{i=0}^{n-1}d(f^i(x),f^i(y))$.
A set $S\ssq X$ is called {\em $(f,\delta,n)$-separated}, for $\delta>0$ and
$n\in\N$, if $d_n(x,y)\geq \delta$ for all $x\neq y\in S$.
Let $\wh S(f,\delta,n)$ denote the maximal cardinality of an
$(f,\delta,n)$-separated set.
Then topological entropy, defined as 
\[
	\htop(f) \ := \ \sup_{\delta>0} \nLim \frac{\log\wh S(f,\delta,n)}{n} \ ,
\]
measures the exponential growth of these numbers, see for example \cite{Walters1982}
for more information.
If topological entropy is zero, then {\em power entropy} instead simply measures
the polynomial growth rate, given by 
\[
	h_{\textrm{pow}}(f) \ := \ \sup_{\delta>0} \varlimsup_{n\to\infty}
		\frac{\log \wh S(f,\delta,n)}{\log n} \ .
\]
We refer to \cite{HasselblattKatok2002HandbookPrincipalStructures} and
\cite{Marco2013} for a more detailed discussion. 

Now, note that already one wandering point is enough to ensure that power
entropy is at least bigger than one \cite{Labrousse2013}.
Given a Morse-Smale homeomorphism on a compact metric space, we hence conclude
that the corresponding power entropy is positive, as claimed above.

This shows that we may have $h_{\textrm{pow}}(f)>\fsc(f)$.
Conversely, consider the map $f:\torus\to\torus,\ (x,y)\mapsto (x,x+y)$ where 
$\T^2:=\R^2/\Z^2$.
Then given $z=(x,y)$ and $z'=(x',y')$, we have that 
\[
d_n(z,z') \ \leq \ n|x-x'|+|y-y'| \ , 
\]
which implies that $\wh S(f,\delta,n) \ \leq \ \frac{C\cdot n}{\delta^2}$ for
some constant $C>0$.
Hence, $h_{\textrm{pow}}(f)\leq 1$.
However, at the same time we have that if $x\neq x'$, then $z$ and $z'$ rotate
in the vertical direction with different speeds, and this makes it easy
to show that $\kreis\times\{0\}$ is an $(f,\delta,\nu)$-separated set for
suitable $\delta,\nu>0$, so that $\Sep(f,\delta,\nu)=\infty$.
Hence, we may also have $\fsc(f)>h_{\textrm{pow}}(f)$, showing that no inequality
holds between the two quantities.\smallskip

{\em Modified power entropy} $h_{\textrm{pow}}^*$ is defined in the same way as
power entropy, with the only difference that the metrics $d_n$ in the definition
are replaced by the {\em Hamming metrics}
\[
	d_n^*(x,y) \ := \ \frac{1}{n}\inergsum d(f^i(x),f^i(y)) \ . 
\]
Since $d_n^*\leq d_n$, modified power entropy is always smaller than power
entropy, and it can be shown that for Morse-Smale systems it is always zero.
The same is true, however, for Denjoy examples and Sturmian subshifts, so that
modified power entropy does not seem suitable to detect topological complexity
on the very fine level we are interested in here.
The same example $f(x,y)=(x,x+y)$ as above shows that we may have
$\fsc(f)>h_{\textrm{pow}}^*(f)$.
An example for the opposite inequality is more subtle, but can be made such that
it demonstrates at the same time the non-existence of a variational principle
for the modified power entropy (a question that was left open in
\cite{HasselblattKatok2002HandbookPrincipalStructures}).
It will be contained in the forthcoming note \cite{GroegerJaeger2015ModifiedPowerEntropy}.

\subsection{Besicovitch space}\label{BesicovitchSpace}

In this section, we want to state some basic results concerning amorphic
complexity in the context of symbolic systems.
The corresponding proofs will be included in the forthcoming paper \cite{FG2014},
where amorphic complexity of symbolic systems is studied more systematically. 

Let $A$ be a finite set, $\Sigma_A:=A^{\N_0}$ and $\rho$ the Cantor metric
on $\Sigma_A$ (see Section \ref{BasicExamples}).
For a general continuous map $f:X\to X$ on a compact metric space $X$ and some
$\delta>0$ we can not expect that
$\varlimsup_{n\to\infty}\countsep{n}(f,\delta,\,\cdot\,,\,\cdot)/n$
is a metric (even not a pseudo-metric since the triangle inequality will usually
fail).
However, this changes in the setting of symbolic dynamics.

\begin{proposition}
	We have that $\big(\tilde d_\delta\big)_{\delta \in (0,1]}$, defined as
	\begin{align*}
		\tilde d_{\delta}(x,y)\ :=\ \varlimsup_{n\to\infty}
		\frac{\countsep{n}(\sigma,\delta,x,y)}{n}\quad\textnormal{for}
		\quad x,y\in\Sigma_A \ ,
	\end{align*}
	is a family of bi-Lipschitz equivalent pseudo-metrics.
\end{proposition}

Note that $\tilde d_1$ is usually called the \emph{Besicovitch pseudo-metric}
and it turns out to be especially useful for understanding certain dynamical
behaviour of cellular automata (see, for example, \cite{BlanchardFormentiKurka1997}
and \cite{CattaneoFormentiMargaraMazoyer1997}).

Now, following a standard procedure, we introduce the equivalence relation
\begin{align*}
	x \  \sim y \ : \ \Leftrightarrow \ \tilde d_{\delta}(x,y) \ = \ 0
	\quad\textnormal{for}\quad x,y\in\Sigma_A \ .
\end{align*}
Due to the previous proposition, this relation is well-defined and independent
of the chosen $\delta$.
Denote the corresponding projection mapping by $[\,\cdot\,]$.
We equip $\big[\Sigma_A\big]$ with the metric $d_\delta\left([x],[y]\right):=
	\tilde d_\delta\left(x,y\right)$, $[x]$, $[y]\in\big[\Sigma_A\big]$ for some
$\delta \in (0,1]$ and call $\big(\big[\Sigma_A\big], d_\delta\big)$ the
\emph{Besicovitch space}.
Given a subshift $\Sigma\ssq\Sigma_A$, we also call $[\Sigma]$ the
{\em Besicovitch space associated to $\Sigma$}.
We have the following properties. 

\begin{theorem}[\cite{BlanchardFormentiKurka1997,
    CattaneoFormentiMargaraMazoyer1997}]
	The Besicovitch space $[\Sigma_A]$ is perfect, complete, pathwise connected
	and (topologically) infinite dimensional.
	However, it is neither locally compact nor separable.
\end{theorem}

Note that we can define the shift map on the Besicovitch space as well and that it
becomes an isometry.
Before we proceed, we need to give the definition of box dimension in
general metric spaces $(X,d)$.
The \emph{lower} and \emph{upper box dimension} of a totally bounded subset
$E\subseteq X$ are defined as
\begin{align*}
	\underline\Dim_B(E)\ := \ \varliminf\limits_{\eps\to 0}
		\frac{\log N_\eps(E)}{-\log\eps}
	\quad\textnormal{and}\quad
	\overline\Dim_B(E)\ := \ \varlimsup\limits_{\eps\to 0}
		\frac{\log N_\eps(E)}{-\log\eps} \ ,
\end{align*}
where $N_\eps(E)$ is the smallest number of sets of diameter strictly smaller
than $\eps$ needed to cover $E$. 
If $\underline\Dim_B(E)=\overline\Dim_B(E)$, then we call their common value 
$\Dim_B(E)$ the \emph{box dimension of $E$}.
Further, let $M_\eps(E)$ be the maximal cardinality of an $\eps$-separated subset
of $E$, that is, a set $S\ssq E$ with $d(x,y)\geq\eps$ for all $x\neq y\in S$.
Then one can replace $N_\eps(E)$ by $M_\eps(E)$ in the definition of box
dimension \cite[Proposition 1.4.6]{Edgar1998}.

Now, suppose $(\Sigma,\sigma)$ is a subshift of $(\Sigma_A,\sigma)$.
If $\left.\sigma\right|_{\Sigma}$ has finite separation numbers, we observe for
each $\delta\in (0,1]$ that
\[
	\Sep(\left.\sigma\right|_{\Sigma},\delta,\nu)=M_\nu([\Sigma])
	\quad\textnormal{and}\quad
	\Span(\left.\sigma\right|_{\Sigma},\delta,\nu)=N_\nu([\Sigma])
	\quad\textnormal{in}\quad
	\big(\big[\Sigma_A\big], d_\delta\big)
\]
for all $\nu\in (0,1]$.
This immediately implies 
\begin{proposition}
	Let $\Sigma$ be a subshift of $\Sigma_A$.
	Then
	\begin{itemize}
		\item[(a)] $\left.\sigma\right|_{\Sigma}$ has finite separation numbers if
			and only if $[\Sigma]$ is totally bounded in $\big[\Sigma_A\big]$, and
		\item [(b)] in this setting, 
			$\uac(\left.\sigma\right|_{\Sigma})=\underline\Dim_B([\Sigma])$ and
			$\oac(\left.\sigma\right|_{\Sigma})=\overline\Dim_B([\Sigma])$.
	\end{itemize}
\end{proposition}

This means for example that all regular Toeplitz subshifts $\Sigma$ (see Section
\ref{RegularToeplitzFlows}) have a totally bounded associated Besicovitch space,
using Theorem \ref{t.finite_separation_numbers} (in fact one can show by a more
direct argument that $[\Sigma]$ is even compact), and that we can find regular
Toeplitz subshifts with associated Besicovitch spaces of arbitrarily high box
dimension, see Theorem \ref{t.toeplitz_sharpbound_examples}.

An example of a minimal and uniquely ergodic subshift with zero topological
entropy such that its projection is not totally bounded is the subshift generated
by the shift orbit closure of the well-known Prouhet-Thue-Morse sequence. (See,
for example, \cite{AlloucheShallit1992} for the definition of this sequence and
further information.)
The fact that the projection is not totally bounded follows directly from the
strict positivity of the \emph{aperiodicity measure} of the Prouhet-Thue-Morse
sequence $x$, defined as 
$\inf_{m\in\N} \varliminf_{n\to\infty} S_n(\sigma,1,x,\sigma^m(x))/n$,
see \cite{PritykinUlyashkina2009, MorseHedlund1938}.

\section{Quantitative analysis of almost sure 1-1 extensions of isometries}
\label{QuantitativeAutomorphic}

The aim of this section is to give a quantitative version of the argument in the
proof of Theorem~\ref{t.finite_separation_numbers} in order to obtain an upper
bound for amorphic complexity in this situation.
For the whole section let $X$ and $\Xi$ be compact metric spaces and
$f:X\to X$ an almost sure 1-1 extension of $g:\Xi\to\Xi$, with factor map $h$.
Further, assume that $g$ is a minimal isometry, with unique invariant probability
measure $\mu$.\foot{Note that a minimal isometry is necessarily uniquely ergodic.}
In this case, it is easy to check that the measure of an $\eps$-ball
$B_\eps(\xi)$ does not depend on $\xi\in\Xi$.
For the scaling of this measure as $\eps\to 0$, we have
\begin{lemma}
	In the above situation, we get
	\begin{equation*}
		\varlimsup_{\eps\to 0} \frac{\log\mu(B_\eps(\xi))}{\log \eps}
		\ = \ \overline{\Dim}_B(\Xi)
	\end{equation*}
	for all $\xi\in\Xi$ and the analogous equality holds for the limit inferior.
\end{lemma}
\begin{proof}
	Recall that we can also use $M_\eps(\Xi)$ in the definition of the box
	dimension of $\Xi$ (see Section \ref{BesicovitchSpace}). 
	Let $\hat\mu(\eps):=\mu(B_\eps(\xi))$, where $\xi\in\Xi$ is arbitrary, and
	suppose $S\ssq \Xi$ is an $\eps$-separated subset with cardinality
	$M_\eps(\Xi)$.
	Observe that the $\eps/2$-balls $B_{\eps/2}(\xi)$ with $\xi\in S$ are pairwise
	disjoint. We obtain $1=\mu(\Xi) \geq \sum_{\xi\in S}\hat\mu(\eps/2)$ and
	thus $M_{\eps}(\Xi)\leq 1/\hat\mu(\eps/2)$.
	Hence,
	\begin{eqnarray*}
		\overline{\Dim}_B(\Xi) & = & \varlimsup_{\eps\to 0}
			\frac{\log M_{\eps}(\Xi)}{-\log \eps}
		\ \leq \ \varlimsup_{\eps\to 0} \frac{\log \hat\mu(\eps/2)}{\log \eps}
		\ = \ \varlimsup_{\eps\to 0} \frac{\log\hat\mu(\eps)}{\log \eps} \ .
	\end{eqnarray*} 
	Conversely, the $\eps$-balls $B_{\eps}(\xi)$ with centres $\xi$ in $S$ cover
	$\Xi$, and this easily leads to the reverse inequality.
\end{proof}

By the Minkowski characterisation of box dimension, we have for $E\ssq \Xi$
\begin{equation}
	\label{e.Minkowski}
	\overline{\Dim}_B(E) \ = \ \overline{\Dim}_B(\Xi)
	- \varliminf_{\eps\to 0} \frac{\log\mu(B_\eps(E))}{\log \eps} \ .
\end{equation}
The proof of this fact in the setting above is the same as in Euclidean space,
see, for example, \cite{Falconer2007FractalGeometry}.
We denote by $\eta_{\delta}(\eps)$ the constant given by
Lemma~\ref{l.eta_function} and let
\begin{equation}\label{e.gamma_def}
	\gamma(h) \ := \ \varlimsup_{\delta\to 0}\varlimsup_{\eps\to 0} \ 
	\frac{\log\eta_\delta(\eps)}{\log\eps} \ .
\end{equation}
This is the scaling factor from
Theorem \ref{t.automorphic_quantitative_intro}, which we restate here as 
\begin{theorem}\label{t.automorphic_quantitative}
	Suppose that the upper box dimension of $\Xi$ is finite and strictly
	positive and $\gamma(h)>0$.
	Then under the above assumptions, we have
	\begin{equation}\label{e.automorphic_upper_bound}
		\oac(f) \ \leq \ \frac{\overline\Dim_B(\Xi)\cdot \gamma(h)}
			{\overline\Dim_B(\Xi)-\sup_{\delta>0}\overline{\Dim_B}(E_\delta)}  \ ,
	\end{equation}
	where $E_\delta=\{\xi\in \Xi\mid \diam(h^{-1}(\xi))\geq \delta\}$. 
\end{theorem}
\begin{proof}
	Without loss of generality, we assume that $\gamma(h)$ is finite and fix $\delta>0$.
	Going back to the end of the proof of Theorem~\ref{t.finite_separation_numbers},
	we find that according to its definition the number $N$ in
	\eqref{e.automorphic_Sep_bound} is equal to $M_{\eta_\delta(\eps)}(\Xi)$.
	Thus, we have already shown that if $\nu>\mu(B_\eps(E_\delta))$ for some
	$\eps>0$, then $\Sep(f,\delta,\nu) \leq M_{\eta_\delta(\eps)}(\Xi)$.

	Now, note that $\mu(B_\eps(E_\delta))$ is monotonously decreasing to $0$ as
	$\eps\to 0$.
	For $\nu$ small enough choose $k\in\N$ such that
	$\mu(B_{2^{-k-1}}(E_\delta))<\nu\leq \mu(B_{2^{-k}}(E_\delta))$.
	We obtain
	\begin{eqnarray*}
		\oac(f,\delta) & \leq &	\varlimsup_{k\to\infty}
			\frac{\log M_{\eta_\delta(2^{-k-1})}(\Xi)}{-\log\mu(B_{2^{-k}}(E_\delta))} \\
		& = & \varlimsup_{k\to\infty}
			\frac{M_{\eta_\delta(2^{-k-1})}(\Xi)}{-\log\eta_\delta(2^{-k-1})} 
			\cdot \frac{\log\eta_\delta(2^{-k-1})}{\log 2^{-k-1}} 
			\cdot \frac{\log 2^{-k-1}}{\log\mu(B_{2^{-k}}(E_\delta))}
		\\ & \leq & \overline\Dim_B(\Xi)\cdot \gamma(h)
			\cdot\left(\varliminf_{k\to\infty}
			\frac{\log\mu(B_{2^{-k}}(E_\delta))}{\log 2^{-k}}\right)^{-1} \\ 
		& = & \frac{\overline\Dim_B(\Xi)\cdot \gamma(h)}
			{\overline\Dim_B(\Xi)-\overline{\Dim}_B(E_\delta)} \ ,
	\end{eqnarray*}
	where we use \eqref{e.Minkowski} for the last equality.
	Taking the supremum over all $\delta>0$ yields
	\eqref{e.automorphic_upper_bound}. 
\end{proof}

\section{Regular Toeplitz flows}\label{RegularToeplitzFlows}

Inspired by earlier constructions of almost periodic functions by Toeplitz, the
notions of Toeplitz sequences and Toeplitz subshifts or flows were introduced by
Jacobs and Keane in 1969 \cite{JacobsKeane1969}.
In the sequel, these systems have been used by various authors to provide a
series of interesting examples of symbolic dynamics with intriguing dynamical
properties, see for example \cite{MarkleyPaul1979PositiveEntropyToeplitzFlows,
Williams1984ToeplitzFlows} or \cite{Downarowicz2005} and references therein.
In what follows, we will study the amorphic complexity for so-called regular
Toeplitz subshifts.

Let $A$ be a finite alphabet, $\Sigma_A=A^\I$ with $\I=\N_0$ or $\Z$ and $\rho$ the 
Cantor metric on $\Sigma_A$ (see Section \ref{BasicExamples}).
Assume that $\omega \in \Sigma_A$ is a non-periodic Toeplitz sequence with
associated Toeplitz subshift $(\Sigma_\omega,\sigma)$, as defined in
Section~\ref{Intro}.
Given $p\in\N$ and $x=(x_k)_{k\in\I}\in\Sigma_A$, let
\[
	\Per(p,x):=\{k\in\I\;|\;x_{k}=x_{k+p\ell}
		\textnormal{ for all }\ell\in\N\} \ .
\]
We call the $p$-periodic part of $\omega$ the \emph{$p$-skeleton of $\omega$}.
To be more precise, define the $p$-skeleton of $\omega$, denoted by $S(p,\omega)$,
as the sequence obtained by replacing $\omega_k$ with the new symbol `$\ast$'
for all $k\notin\Per(p,\omega)$.
Note that the $p$-skeletons of two arbitrary points in $\Sigma_{\omega}$ coincide
after shifting one of them by at most $p-1$ positions.
We say that $p$ is an \emph{essential period of $\omega$} if $\Per(p,\omega)$ is
non-empty and does not coincide with $\Per(\tilde p,\omega)$ for any $\tilde p<p$.
A \emph{weak periodic structure of $\omega$} is a sequence
$(p_{\ell})_{\ell\in\N}$ such that each $p_{\ell}$ divides $p_{\ell+1}$ and
\begin{align}\label{Toeplitz_periods}
  \bigcup\limits_{\ell\in\N}\Per(p_{\ell},\omega)=\I \ .
\end{align}
If, additionally, all the $p_l$'s are essential, we call $(p_{\ell})_{\ell\in\N}$
a \emph{periodic structure of $\omega$}.
For every (non-periodic) Toeplitz sequence we can find at least one periodic
structure \cite{Williams1984ToeplitzFlows}.

\begin{remark}\label{remark_non_essential_ps}
	Note that from each weak periodic structure we can obtain a periodic
	structure in the following way.
	Suppose $(p_{\ell})_{\ell\in\N}$ is a weak periodic structure of $\omega$.
	Without loss of generality, we can assume that 	
	$\Per(p_{\ell},\omega)\neq\emptyset$ and
	$\Per(p_{\ell},\omega)\subsetneq\Per(p_{\ell+1},\omega)$ for all $\ell\in\N$
	(recall that $\omega$ is non-periodic).
	For each $p_{\ell}$ choose the smallest $\tilde p_{\ell}\in\N$ such that
	$\Per(\tilde p_{\ell},\omega)$ coincides with $\Per(p_{\ell},\omega)$.
	Then by definition $\tilde p_{\ell}$ is an essential period.
	Since $p_{\ell}$ divides $p_{\ell+1}$ we have $\Per(\tilde p_{\ell},\omega)
		\subset\Per(\tilde p_{\ell+1},\omega)$.
	The next lemma and the minimality of the $\tilde p_{\ell}$'s imply that
	$\tilde p_{\ell}$ divides $\tilde p_{\ell+1}$ for each $\ell\in\N$, so that
	$(\tilde p_{\ell})_{\ell\in\N}$ is a periodic structure.
\end{remark}

The next lemma is probably well-known to experts and we omit the proof here.
\begin{lemma}\label{Toeplitz_gcd}
	If $\Per(p,x)\subseteq\Per(q,x)$, then $\Per(\gcd(p,q),x)=\Per(p,x)$ where 
	$x\in\Sigma_A$ and $p,q\in\N$.
\end{lemma}

Given $p\in\N$, we define the relative densitiy of the $p$-skeleton of $\omega$ by
\[
	D(p)\  := \ \frac{\#(\Per(p,\omega)\cap [0,p-1])}{p} \ .
\]
Since $\omega$ is non-periodic, we have $D(p)\leq 1-1/p$.
For a (weak) periodic structure $(p_{\ell})_{\ell\in\N}$, the densities
$D(p_{\ell})$ are non-decreasing in $\ell$ and we say that
$(\Sigma_{\omega},\sigma)$ is a \emph{regular Toeplitz subshift} if
$\lim_{\ell\to\infty}D(p_{\ell})=1$.
Note that regularity of a Toeplitz subshift does not depend on the chosen (weak)
periodic structure (use \eqref{Toeplitz_periods} and Lemma \ref{Toeplitz_gcd}).

It is well-known that a regular Toeplitz subshift is an almost sure 1-1
extension of a minimal isometry (an odometer) \cite{Downarowicz2005}.
Thus, we obtain from Theorem~\ref{t.finite_separation_numbers} that its asymptotic
separation numbers are finite.
However, as mentioned in the introduction, a quantitative analysis is possible
and yields the following.
\begin{theorem} \label{t.toeplitz_estimate}
	Suppose $(\Sigma_{\omega},\sigma)$ is a regular Toeplitz subshift and let
	$(p_{\ell})_{\ell\in\N}$ be a (weak) periodic structure of $\omega$.
	For $\delta,s>0$ we have
	\[
		\varlimsup\limits_{\nu\to 0}
			\frac{\Sep(\sigma,\delta,\nu)}{\nu^{-s}}
		\ \leq \  C\cdot\varlimsup\limits_{\ell\to\infty}
			\frac{p_{\ell+1}}{(1-D(p_{\ell}))^{-s}} \ ,
	\]
	with $C=C(\delta,s)>0$.
\end{theorem}
Note that this directly implies Theorem~\ref{t.toeplitz}.
\begin{proof}
	Recall that since $\omega$ is a regular Toeplitz sequence, the densities 
	$D(p_{\ell})$ are non-decreasing and converge to $1$.
	Choose $m\in\N$ with $2^{-m}<\delta\leq 2^{-m+1}$ and $\ell\in\N$ such that
	\begin{align}
		\label{toeplitz_upper_bound_cond_nu}
		(2m+1)2(1-D(p_{\ell+1}))\ < \ \nu\ \leq\ (2m+1)2(1-D(p_{\ell})) \ .
	\end{align}
	Then we have
	\[
		\Sep(\sigma,\delta,\nu)\ \leq \ \Sep(\sigma,2^{-m},(2m+1)2(1-D(p_{\ell+1})))
	\]
	and claim that the second term is bounded from above by $p_{l+1}$.
	
	Assume for a contradiction that there exists a 
	$(\sigma, 2^{-m},(2m+1)2(1-D(p_{\ell+1})))$-separated set
	$S\subseteq\Sigma_{\omega}$	with more than $p_{\ell+1}$ elements.
	Then, there are at least two points $x=(x_k)_{k\in\I}$,
	$y=(y_k)_{k\in\I}\in S$	with the same $p_{\ell+1}$-skeleton.
	This means $x$ and $y$ can differ at most at the remaining positions
	$k\notin\Per(p_{\ell+1},x)=\Per(p_{\ell+1},y)$.
	Using the fact that $\rho(x,y)\geq 2^{-m}$ if and only if $x_k\neq y_k$ for
	some $k\in\I$ with $|k|\leq m$, we obtain
	\begin{eqnarray*}
		\lefteqn{ \varlimsup\limits_{n\to\infty}\frac{\countsep{n}(\sigma,2^{-m},x,y)}{n}  
		\ \leq \  (2m+1)\varlimsup\limits_{n\to\infty}
			\frac{\#\left\{0\leq k<n\;|\;x_k\neq y_k\right\}}{n} } \\
		& \leq & (2m+1)\varlimsup\limits_{n\to\infty}
			\frac{\#([0,n-1]\smin \Per(p_{\ell+1},\omega))}{n} 
		\ = \ (2m+1)(1-D(p_{\ell+1})) \ . 
	\end{eqnarray*}
	However, this contradicts \eqref{toeplitz_upper_bound_cond_nu}.
	Hence, we obtain
	\[
		\frac{\Sep(\sigma,\delta,\nu)}{\nu^{-s}}
		\ \leq \ C(\delta,s)\cdot\frac{p_{\ell+1}}{(1-D(p_{\ell}))^{-s}} \ ,
	\]
	where $C(\delta,s):=(2m+1)^s$. Note that $m$ only depends on $\delta$.
	Taking the limit superior yields the desired result.
\end{proof}

For the remainder of this section, our aim is to provide a class of examples
demonstrating that the above estimate is sharp and that the amorphic complexity
of regular Toeplitz flows takes at least a dense subset of values in
$[1,\infty)$.
To that end, we first recall an alternative definition of Toeplitz sequences \ 
(cf.\ \cite{JacobsKeane1969}).
Consider the extended alphabet $\mathcal A:=A\cup\{\ast\}$ where we can think of $\ast$
as a hole or placeholder like in the definition of the $p$-skeleton. 
Then, $\omega\in\Sigma_A$ is a Toeplitz sequence if and only if there exists an
\emph{approximating sequence} $(\omega^{\ell})_{\ell\in\N}$ of periodic points
in $(\Sigma_{\mathcal A},\sigma)$ such that (i) for all $k\in\I$ we have
$\omega_k^{\ell+1}=\omega_k^{\ell}$ as soon as $\omega_k^{\ell}\in A$ for some
$\ell\in\N$ and (ii) $\omega_k=\lim_{\ell\to\infty}\omega_k^{\ell}$, see
\cite{Eberlein1971}.
Such an approximating sequence of a Toeplitz sequence is not unique.
For example, every sequence of $p_{\ell}$-skeletons $(S(p_{\ell},\omega))_{\ell\in\N}$
with $(p_{\ell})_{\ell\in\N}$ a (weak) periodic structure satisfies these properties.

Let us interpret Theorem~\ref{t.toeplitz_estimate} in this context. For a
$p$-periodic point $x\in\Sigma_{\mathcal A}$, we can define the relative density
of the holes in $x$ by
\[
	r(x) \ := \ \frac{\#\{0\leq k<p\;|\; x_k=\ast\}}{p} \ .
\]
Note that $D(p)=1-r\left(S(p,\omega)\right)$ for every $p\in\N$.
Suppose $(\omega^{\ell})_{\ell\in\N}$ is an approximating sequence of $\omega$.
We say $(p_{\ell})_{\ell\in\N}$ is a \emph{sequence of corresponding periods of
$(\omega^{\ell})_{\ell\in\N}$} if $p_{\ell}$ divides $p_{\ell+1}$ and
$\sigma^{p_{\ell}}(\omega^{\ell})=\omega^{\ell}$ for each $\ell\in\N$.
We have that $r(\omega^{\ell})\geq 1/p_{\ell}$.  Moreover, $r(\omega^{\ell}) \geq
1-D(p_\ell)$, so that Theorem~\ref{t.toeplitz_estimate} implies
\begin{corollary}
	\label{cor_upper_bound_ham}
	Assume $(\Sigma_{\omega},\sigma)$ is a regular Toeplitz subshift.
	Let $(\omega^{\ell})_{\ell\in\N}$ be an approximating sequence of $\omega$
	and let $(p_{\ell})_{\ell\in\N}$ be a sequence of corresponding periods
	of $(\omega^{\ell})_{\ell\in\N}$.
	Furthermore, assume $p_{\ell+1}\leq C p_{\ell}^t$ and 
	$r(\omega^{\ell})\leq K/p_{\ell}^u$ for $\ell$ large enough, where 
	$C,t\geq 1$, $u\in (0,1]$ and $K>0$.
	Then
	\[
		\oac(\sigma)\leq\frac{t}{u} \ .
	\]
\end{corollary}

For the construction of examples, it will be convenient to use so-called
$(p,q)$-Toeplitz (infinite) words, as introduced in \cite{CassaigneKarhumaeki1997}. 
Let $\I=\N_0$.
Suppose $v$ is a finite and non-empty word with letters in $\mathcal A$ and at least
one entry distinct from $\ast$. 
Let $\abs{v}$ be its length and $\abs{v}_{\ast}$ be the number of holes in $v$.
We use the notation $\overline v\in\Sigma_{\mathcal A}$ for the one-sided periodic
sequence that is created by repeating $v$ infinitely often.
Define the sequence $(T_{\ell}(v))_{\ell\in\N}$ recursively by
\[
	T_{\ell}(v) \ := \ F_{v}(T_{\ell-1}(v)) \ ,
\]
where $T_0(v):=\overline\ast$ and $F_{v}:\Sigma_{\mathcal A}\to\Sigma_{\mathcal A}$
assigns to each $x\in\Sigma_{\mathcal A}$ the sequence that is obtained from
$\overline v$ by replacing the subsequence of all occurrences of $\ast$ in
$\overline v$ by $x$.
We get that $(T_{\ell}(v))_{\ell\in\N}$ is an approximating sequence and denote the
corresponding Toeplitz sequence by $T(v)$ \cite{CassaigneKarhumaeki1997}.
Setting $p:=\abs{v}$, $q:=\abs{v}_{\ast}$ and $d:=\gcd(p,q)$, we say $T(v)$ is a
\emph{$(p,q)$-Toeplitz word}.
One particular nice feature of $(p,q)$-Toeplitz words is that in order to exclude
periodicity one only has to check a short prefix of the sequence.

\begin{theorem}[{\cite[Theorem 4]{CassaigneKarhumaeki1997}}]
	\label{thm_periodicity_Toeplitz_words}
	Let $T(v)$ be a $(p,q)$-Toeplitz word.
	Then $T(v)$ is periodic if and only if its prefix of length $p$ is $d$-periodic.
\end{theorem}

\begin{theorem}\label{t.toeplitz_sharpbound_examples}
	Suppose $m\in\N$ and let $0{\,^m}1$ be the word starting with $m$ zeros and
	ending with a single one.
	Furthermore, let $v$ be a word with letters in $\mathcal A=\{0,1,\ast\}$ such
	that $1\leq\abs{v}_{\ast}\leq\abs{v}\leq m$.
	Then $\omega:=T(0{\,^m}1v)$ is a $(p,q)$-Toeplitz word and the corresponding
	regular	Toeplitz subshift $(\Sigma_{\omega},\sigma)$ has amorphic complexity
	\[
		\fsc(\sigma)=\frac{\log p/d}{\log p/q} \ .
	\]
\end{theorem}
\begin{proof}
	Define for each $n\in\N$ and $x=(x_k)_{k\in\N_0}$, 
	$y=(y_k)_{k\in\N_0}\in\Sigma_{\mathcal A}$
	\[
		S_{n}(x,y) \ := \ \#\left\{0\leq k<n\;|\;x_k, y_k\neq\ast\textnormal{ and }
			x_k\neq y_k\right\} \ .
	\]
	Observe that
	\begin{align*}
		\begin{split}		
			\lefteqn{S_{p}\big(T(0{\,^m}1v),\sigma^{j}(T(0{\,^m}1v))\big)}\\  
			&\quad\geq \ S_{p}\big(T_{1}(0{\,^m}1v),	\sigma^{j}(T_{1}(0{\,^m}1v))\big) \ 
			= \ S_{p}\big(\overline{0{\,^m}1v},\sigma^{j}(\overline{0{\,^m}1v})\big)
			\ \geq \ 1
		\end{split}
	\end{align*}	
	for every $0<j<p$ due to the special form of the prefix $0{\,^m}1$ and the
	assumption $\abs{v}\leq m$.
	This directly implies that $\omega$ is non-periodic, using Theorem 
	\ref{thm_periodicity_Toeplitz_words}.
	
	To get an upper bound for $\oac(\sigma)$, note that 
	$(p^{\ell}/d^{\ell-1})_{\ell\in\N}$ is a sequence of corresponding periods of
	$(T_{\ell}(0{\,^m}1v))_{\ell\in\N}$ and
	$r(T_{\ell}(0{\,^m}1v))=q^{\ell}/p^{\ell}$ for each $\ell\in\N$.
	This is proved easily by induction: The statement is true for
	$T_1(0{\,^m}1v) = \overline{0{\,^m}1v}$.
	When going from $\ell$ to $\ell+1$, by the induction hypothesis each of the
	$p^\ell/d^{\ell-1}$-periodic blocks of $T_\ell(0{\,^m}1v)$ has 
	$q^{\ell}/d^{\ell-1}$ free positions.
	In order to accommodate $q/d$ such periodic blocks of $T_\ell(0{\,^m}1v)$ it
	needs $p^{\ell}/d^\ell$ of the $p$-periodic blocks of $\overline{0{\,^m}1v}$
	with $q$ free positions each.
	Thus, the resulting periodic block of $T_{\ell+1}(0{\,^m}1v)$ has length
	$p^{\ell+1}/d^\ell$ and $q^{\ell+1}/d^{\ell}$ free positions.
	Now, Corollary \ref{cor_upper_bound_ham} gives the desired upper bound.

	In order to prove the lower bound, we show by a similar induction that
	\begin{align}\label{0m1_lb}
		S_{p^{\ell}/d^{\ell-1}}\big(T_{\ell}(0{\,^m}1v),
			\sigma^{j}(T_{\ell}(0{\,^m}1v))\big) \ \geq \  q^{\ell-1}/d^{\ell-1}
	\end{align}
	for every $0<j<p^{\ell}/d^{\ell-1}$ and $\ell\in\N$.
	If $j$ is not a multiple of $p$, then by induction assumption each
	$p^\ell/d^{\ell-1}$-periodic block of $T_\ell(0{\,^m}1v)$ has
	$p/d\cdot q^{\ell-2}/d^{\ell-2}$ mismatches with $\sigma^j(T_\ell(0{\,^m}1v))$
	coming from the mismatches of the $p/d$ contained $p^{\ell-1}/d^{\ell-2}$-periodic
	blocks of $T_{\ell-1}(0{\,^m}1v)$ with $\sigma^j(T_{\ell-1}(0{\,^m}1v))$.
	If $j$ is a multiple of $p$, then the mismatches result in a similar way
	from the shift in the sequences that are inserted into $\overline{0{\,^m}1v}$,
	since $\sigma^{ip}(T_\ell(0{\,^m}1v))=F_v(\sigma^{iq}(T_{\ell-1}(0{\,^m}1v)))$.
	Note that the fact that $p^{\ell}/d^{\ell-1}$ is a minimal period comes from
	the	assumption that $d=\gcd(p,q)$.

	As a direct consequence from \eqref{0m1_lb}, we obtain that for all $\ell\in\N$
	and  $0\leq i<j<p^{\ell}/d^{\ell-1}$
	\[
		S_{p^{\ell}/d^{\ell-1}}\big(\sigma^{i}(T_{\ell}(0{\,^m}1v)),
			\sigma^{j}(T_{\ell}(0{\,^m}1v))\big)\geq q^{\ell-1}/d^{\ell-1} \ .
	\]
	Hence,
	\[
		\{\omega,\sigma(\omega),\dots,\sigma^{p^{\ell}/d^{\ell-1}-1}(\omega)\}
	\]
	is a $(\sigma,1,q^{\ell-1}/p^{\ell})$-separated set.
	For $\nu$ small enough choose $\ell\in\N$ such that $q^{\ell}/p^{\ell+1}<\nu\leq
	q^{\ell-1}/p^{\ell}$ and observe
	\[
		\frac{\Sep(\sigma,\delta,\nu)}{\nu^{-s}}
		\ \geq \ \frac{\Sep(\sigma,1,q^{\ell-1}/p^{\ell})}{\nu^{-s}}
		>\frac{p^{\ell}}{d^{\ell-1}}\cdot\frac{q^{ls}}{p^{(l+1)s}}
	\]
	for $\delta, s>0$.  This yields $\uac(\sigma)\geq(\log p/d)/(\log p/q)$.
\end{proof}

As the set $\{\log p/\log (p/q)\mid p,q\in\N,\, \gcd(p,q)=1\}$ is dense in
$[1,\infty)$, we obtain
\begin{corollary}\label{c.toeplitz_densevalues}
	In the class of $(p,q)$-Toeplitz words, amorphic complexity takes (at least)
	a dense set of values in $[1,\infty)$.
\end{corollary}

\begin{remark}
	From the results in \cite[Theorem 5]{CassaigneKarhumaeki1997}, one can
	directly conclude that for all (non-periodic) $(p,q)$-Toeplitz words the
	power entropy equals $(\log p/d)/(\log p/q)$.
	Thus, for our examples provided by the last theorem power entropy and
	amorphic complexity coincide.
	It would be interesting to know if this is true for all $(p,q)$-Toeplitz
	words, or if not, in which cases this equality holds.
\end{remark}

\section{Strange non-chaotic attractors in pinched skew products}
\label{PinchedSystems}

As we have mentioned in previous sections, one of the main reasons for
considering amorphic complexity is the fact that it gives value zero to
Morse-Smale systems, while power entropy assigns a positive value to these.
The latter is unsatisfactory from an abstract viewpoint, since such dynamics
should certainly be considered entirely trivial.
At the same time, however, this issue may also raise practical problems.
In more complicated systems, attractor-repeller dynamics may coexist with other
more subtle dynamical mechanisms.
In this case, the contribution of the Morse-Smale component to power entropy may
overlay other effects, and two systems may not be distinguishable despite a
clearly different degree of dynamical complexity.

Of course, the computation of topological complexity invariants in more complex
non-linear dynamical systems will generally be difficult and technically involved.
Nevertheless, we want to include one classical example in this section which fits
the situation described above.
In order to keep the exposition brief, we concentrate on a positive qualitative
result for amorphic complexity and refrain from going into detail concerning
(modified) power entropy.\smallskip

Recall that $\T^1=\R/\Z$, $d$ is the usual metric on $\T^1$ and $\Leb$ denotes
the Lebesgue measure on $\T^1$ (cf.\ Section \ref{BasicExamples}).
Suppose $f:\kreis\times[0,1]\to\kreis\times[0,1]$ is a continuous map of the form
\begin{equation}
  \label{eq:1}
  f(\theta,x) \ = \ (\theta+\omega\mod 1,f_\theta(x)) \ ,
\end{equation}
where $\omega\in\kreis$ is irrational.
For the sake of simplicity we will suppress `$\textrm{mod } 1$' in the following.
The maps $f_\theta:[0,1]\to[0,1]$ are called {\em fibre maps}, $f$ itself is
often called a {\em quasiperiodically forced (qpf) 1D map}.
If all the fibre maps in (\ref{eq:1}) are monotonically increasing, then the
topological entropy of $f$ is zero.\footnote{This is a direct consequence of
\cite[Theorem 17]{bowen:1971}.}
Notwithstanding, systems of this type may exhibit considerable dynamical
complexity.
A paradigm example in this context are so-called {\em pinched skew products},
introduced by Grebogi et al in \cite{grebogi/ott/pelikan/yorke:1984} and later
treated rigorously by Keller \cite{keller:1996}.
In order to fix ideas, we concentrate on the specific parameter family
\begin{equation}
  \label{eq:2}
  f(\theta,x) \ = \ (\theta+\omega,\tanh(\alpha x)\cdot \sin(\pi \theta)) \ ,
\end{equation}
which is close to the original example introduced by Grebogi and his coworkers.
The crucial features of this system are that 
\romanlist
	\item the zero line $\kreis\times\{0\}$ is $f$-invariant;
	\item the fibre maps $f_\theta:x\mapsto \tanh(\alpha x)\cdot
		\sin(\pi\theta)$ are all concave;
	\item the fibre map $f_{0}$ sends the whole interval $[0,1]$ to $0$.
\listend
Item (iii) is often refered to as {\em pinching}.
It is the defining property of the general class of pinched skew products, as
introduced in \cite{glendinning:2002}.
Note that all of the arguments and statements in this section immediately carry
over to a whole class of fibre maps and higher-dimensional rotations in the base
(see the set $\mathcal T^\ast$ and Example 4.1 in \cite{GroegerJaeger2013}).

A function $\varphi:\kreis\to[0,1]$ is called an {\em invariant graph} of
(\ref{eq:1}) if $f_\theta(\varphi(\theta))=\varphi(\theta+\omega)$ for all
$\theta\in\kreis$.
In this case, the associated point set $\Phi:=\{(\theta,\varphi(\theta))\mid
	\theta\in\kreis\}$ is $f$-invariant.\footnote{Slightly abusing notation,
	the term {\em invariant graph} is used both for the function and its graph.}
If the fibre maps are all differentiable, the {\em (vertical) Lyapunov exponent}
of $\varphi$ is defined as 
\begin{equation}\label{eq:3}
	\lambda(\varphi) \ := \ 
	\int_{\kreis} \log|f'_\theta(\varphi(\theta))| \ d\theta \ .
\end{equation}
If $\lambda(\varphi)\leq 0$, then $\Phi$ is an attractor in the sense of Milnor
\cite{milnor:1985} (see, for example, \cite[Proposition 3.3]{jaeger:2003}).
If $\varphi$ is continuous, then it is even a topological attractor and contains
an open annular neighbourhood in its basin of attraction.
In case $\varphi$ is not continuous, the attractor $\Phi$ combines non-chaotic
dynamics (zero entropy, absence of positive Lyapunov exponents) with a
complicated topological structure (related to the absence of continuity, see
\cite{stark:2003,jaeger:2007} for more information).
Due to this combination of properties, it is called a {\em strange non-chaotic
attractor (SNA)}.

As mentioned above, the zero line $\Phi_0:=\kreis\times\{0\}$ is an invariant
graph of (\ref{eq:2}).
An elementary computation yields $\lambda(\varphi_0) = \log \alpha - \log 2$.
If $\alpha\leq 2$, so that $\lambda(\varphi_0)\leq 0$, then $\Phi_0$ is the global
attractor of the system, meaning that
$\Phi_0=\bigcap_{n\in\N} f^n(\kreis\times[0,1])$.
Accordingly, all orbits converge to the zero line, that is, $\nLim f^n_\theta(x)=0$,
where $f^n_\theta=f_{\theta+(n-1)\omega}\circ\ldots\circ f_\theta$.
If $\alpha>2$, then this picture changes drastically.
Now, a second invariant graph $\varphi^+$ with negative Lyapunov exponent appears,
which satisfies $\varphi^+(\theta)>0$ for $\Leb$-a.e.\ $\theta\in\kreis$
\cite{keller:1996}.
However, at the same time there exists a dense set of $\theta$'s with
$\varphi^+(\theta)=0$.
This latter fact is easy to see, since $\varphi^+$ is invariant and 
$f_{0}(\varphi^+(0))=0$ by property (iii) above.
Thus, $\varphi^+$ is an SNA (see Figure~\ref{fig: upper bounding graphs}(a)).
\begin{figure}
	\centering 
	\subfloat[]{\includegraphics[width=2.6in, height=2in]{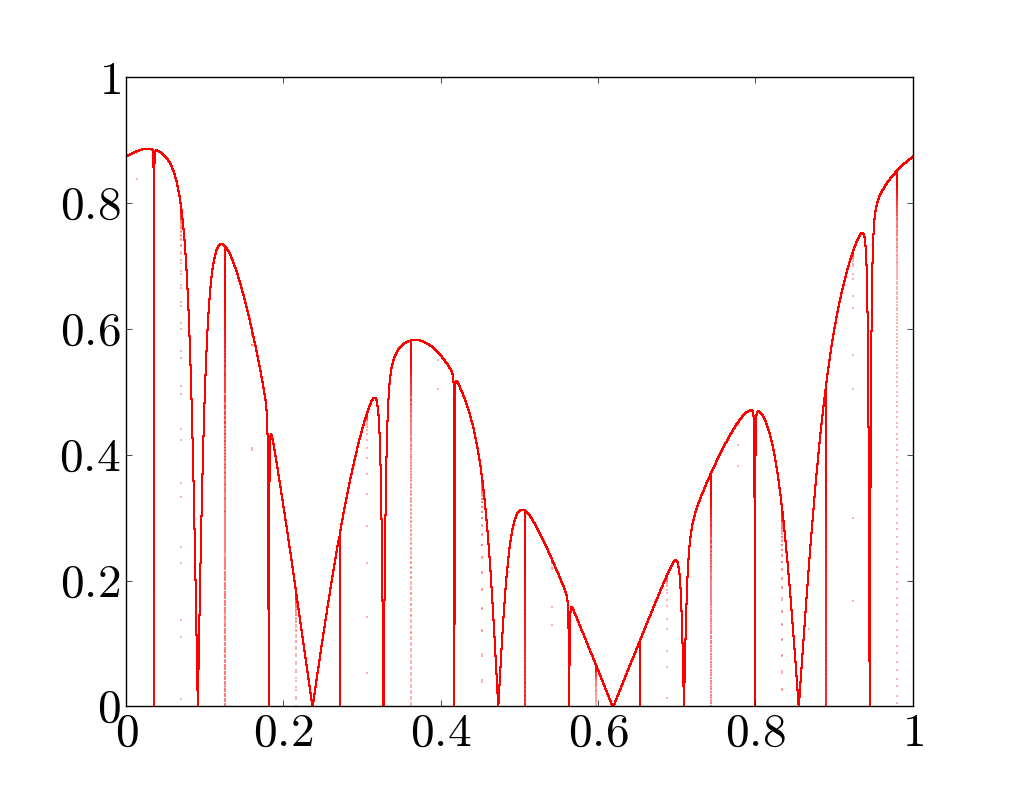}}
	\subfloat[]{\includegraphics[width=2.6in, height=2in]{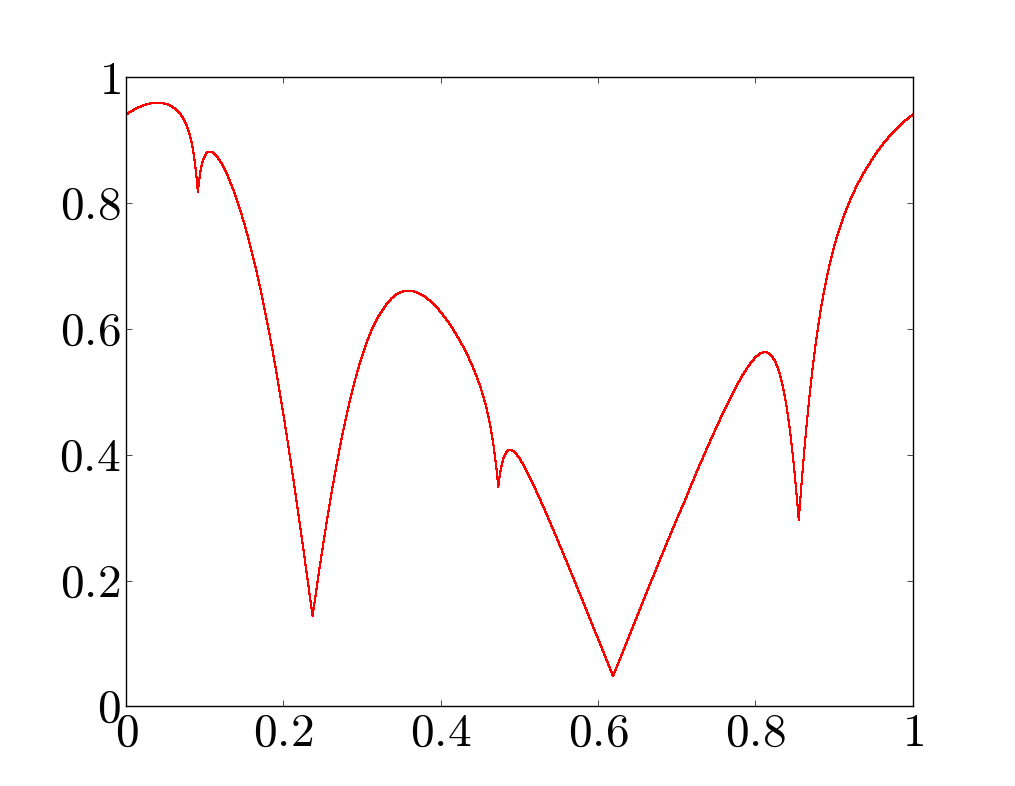}}
	\caption{The upper bounding graphs of $f$ (a) and $f^\eps$ (b) with $\eps=0.05$.
		In both cases, $\omega$ is the golden mean and $\alpha=3$.
		The horizontal axis is $\T^1$, the vertical axis $[0,1]$.} 
	\label{fig: upper bounding graphs}
\end{figure}
We note that $\varphi^+$ can be defined as the upper bounding graph of the global
attractor
$\mathcal{A}:=\bigcap_{n\in\N} f^n(\kreis\times[0,1])$, that is,
\begin{equation}\label{e.upper_bounding_graph}
	\varphi^+(\theta) \ := \ \sup\{x\in[0,1]\mid (\theta,x)\in\mathcal{A}\} \ .
\end{equation}

The rigorous proof of these facts in \cite{keller:1996} is greatly simplified by
the particular structure of pinched skew products.
More natural systems, though, often occur as the time-one-maps of flows generated
by scalar differential equations with quasiperiodic right-hand side.
In particular, this means that such systems are invertible and the non-invertible
pinched skew products have a certain toy-model character.
Nowadays, however, established methods of multiscale analysis yield a wealth of
results about the existence and structure of SNA in broad classes of invertible
systems as well \cite{young:1997,bjerkloev:2005,bjerkloev:2007,jaeger:2006,fuhrmann:2014}.
In many cases, it turned out that this machinery allows to transfer results
and insights first obtained for pinched systems to a more general setting.
One example is the computation of the Hausdorff dimension of SNA 
\cite{GroegerJaeger2013,FGJ2014DimensionsSNA}, another is a question about the
structure of their topological closure ({\em filled-in property})
\cite{jaeger:2007,bjerkloev:2007,FGJ2014DimensionsSNA}, going back to Herman
\cite{herman:1983}.
In this sense, pinched systems have proven to be very adequate models for more
general qpf systems.\medskip

If (\ref{eq:2}) is slightly modified by adding a small positive constant $\eps>0$
to the multiplicative forcing term $\sin(\pi\theta)$, we obtain a new system
\begin{equation}\label{eq:4}
	f^\eps(\theta,x) \ = \ (\theta+\omega,\tanh(\alpha x)
		\cdot(\sin(\pi\theta) + \eps)) \ .
\end{equation}
In this case, the Lyapunov exponent $\lambda(\varphi_0)$ still increases strictly
with $\alpha$ and there exists a critical value 
$\alpha_c=\int_{\kreis}\log|\sin(\pi\theta)+\eps|\ d\theta$ at which
$\lambda(\varphi_0)=0$.
If $\alpha\leq \alpha_c$, the graph $\Phi_0$ is the global attractor as before,
and if $\alpha>\alpha_c$ a second invariant graph $\varphi^+$ with negative
Lyapunov exponent appears above $\Phi_0$.
However, there is one important qualitative difference to the previous situation.
Due to the invertibility of the system, it is easy to show that the graph
$\Phi^+$ is a continuous curve (see Figure~\ref{fig: upper bounding graphs}(b)).

The resulting dynamics are much simpler than in the case of an SNA.
In particular, the system is conjugate to the direct product of the underlying
irrational rotation with a Morse-Smale map $g$ on $[0,1]$, with unique repelling
fixed point $0$ and unique attracting fixed point $x\in(0,1)$, and all points
outside $\Phi_0$ are Lyapunov stable.
In contrast to this, the system (\ref{eq:2}) with $\alpha>2$ has sensitive
dependence on initial conditions \cite{glendinning/jaeger/keller:2006}, and thus
has no Lyapunov stable points at all.

It is thus reasonable to expect that both cases can be distinguished by means of
a suitable topological complexity invariant.
However, in this case the rigorous analysis is more difficult than in the
previous chapters.
The reason is that while the qualitative analysis of pinched skew products is
comparatively easy due to their particular structure, a more detailed
quantitative study is still rather involved on a technical level.
In particular, it typically requires to exclude an exceptional set of measure
zero from the considerations, on which the dynamics are hard to control.
Since topological complexity invariants in the zero entropy regime typically do
not satisfy a variational principle (see Introduction), the lack of control even
on a set of measure zero impedes their computation.
For this reason, we do not attempt to determine the power entropy or modified
power entropy of (\ref{eq:2}) in a rigorous way.
However, based on the intuition gained from previous work on pinched systems in
\cite{jaeger:2007,GroegerJaeger2013} and heuristic arguments, we expect that the
power entropy of (\ref{eq:2}) with $\alpha>2$ equals 1, whereas the modified
power entropy is zero.
It is easy to show that the same values are attained by (\ref{eq:4}) with
$\alpha>\alpha_c$, and hence both quantities should not be suitable to
distinguish between the two substantially different types of behaviour.
As we have mentioned before, this was one of the original motivations for the
introduction of amorphic complexity.

In principle, though, the same restrictions as for the computation of (modified)
power entropy hold for amorphic complexity, and the existence of the exceptional
uncontrolled set does not allow a straightforward application of the concept.
What we concentrate on here is to show that
amorphic complexity distinguishes between SNAs and continuous attractors.
This is the main result of this section.
Recall that $\omega\in\kreis$ is called {\em Diophantine} if there exist
constants $c,d>0$ such that
\begin{equation}\label{e.Diophantine}
	d(n\omega,0) \ \geq \ cn^{-d}
\end{equation}
for all $n\in\N$.
\begin{theorem}\label{t.pinched_systems}
	Suppose $\omega$ is Diophantine and $\alpha$ in (\ref{eq:2}) is sufficiently
	large.
	Then there exists an invariant (under the rotation by angle $\omega$) set
	$\Omega\ssq\kreis$ of full Lebesgue measure such that
	\[
		0 \ < \ \uac\big(\left.f\right|_{\Omega\times[0,1]}\big)\ \leq \
		\oac\big(\left.f\right|_{\Omega\times[0,1]}\big)\ < \ \infty \ .
	\]
	In contrast to this, we have $\textrm{ac}(f^\eps)=0$ if $f^\eps$ is given by
	(\ref{eq:4}) with $\eps>0$ and any $\alpha\geq 0$.
\end{theorem}
Note that $\textrm{ac}(f^\eps)=0$ follows immediatly from the conjugacy between
$f^\eps$ and the product of the underlying rotation with the Morse-Smale map
$g$ from above, using Corollary~\ref{fsc_top_inv}. 
\begin{remark}\label{r.ac_measure}
	The approach taken by restricting to a subset of full measure in the above
	statement can be formalized in a more systematic way.
	Although we do not pursue this issue much further here, we believe that this
	may make the concept of amorphic complexity applicable to an even broader
	range of systems.
	Let $(X,d)$ be a metric space and consider a map $f:X\to X$ and let $E\ssq X$.
	For the definition of $\Sep_{E}(f,\delta,\nu)$ see 
	\eqref{e.subset_separation_numbers}.
 	Further, we say $A\ssq E$ is \emph{$(f,\delta,\nu)$-spanning in $E$} if
	for each $x \in E$ there is $y\in A$ such that
	$\varlimsup_{n\to \infty}\frac{1}{n}{S_n(f,\delta,x,y)}<\nu$.
	Let $\Span_{E}(f,\delta,\nu)$ be the smallest cardinality of any
	$(f,\delta,\nu)$-spanning set in $E$.
	For the definition of $\uac_{E}(f)$ and $\oac_{E}(f)$ see
	\eqref{eq: defn h am for subsets} and note that similarly as before
	we can use $\Span_{E}(f,\delta,\nu)$ instead of $\Sep_{E}(f,\delta,\nu)$
	there (see Section~\ref{SpanningSets}).
	
	Now, suppose we are given a Borel probability measure $\mu$ on $X$.
	Then we define
	\begin{align*}
		\uac_{\mu}(f)&\ := \ \inf\left\{\uac_{E}(f) \mid E\ssq X
			\textrm{ and }\mu(E)=1\right\} \ ,\\
		\oac_\mu(f)&\ := \ \inf\left\{\oac_E(f) \mid  E\ssq X
			\textrm{ and }\mu(E)=1\right\} \ .
	\end{align*}

	With these notions, what we actually show is that under the assumptions of
	Theorem~\ref{t.pinched_systems} we have
	\begin{align*}
		0<\uac_{\mu}(f) \leq \oac_{\mu}(f) < \infty,	 
	\end{align*}
	where $\mu$ can be either the
	Lebesgue measure on $\kreis\times[0,1]$, or the measure $\mu_{\phi^+}$,
	which is the Lebesgue measure on $\kreis$ lifted to the graph $\Phi^+$, 
	i.e.\ $\mu_{\phi^+}(A):=\Leb(\pi_{\theta}(A\cap\Phi^+))$ where
	$A\subseteq\kreis\times[0,1]$ is Borel measurable and $\pi_\theta$ is the
	projection onto the first coordinate.
	Note that these statements are slightly stronger than the ones given in
	Theorem~\ref{t.pinched_systems}.
\end{remark}

We first consider the lower bound.
Thereby, we will focus on the SNA $\Phi^+$ and show that the restriction of $f$
to this set already has positive lower amorphic complexity.
\begin{proposition}\label{prop: ac > 0}
	Suppose $\omega$ is Diophantine and $\alpha$ in (\ref{eq:2}) is
    sufficiently large.
    Then there is a positive uniform lower bound for
    $\uac_{\Phi^+\cap(\Omega\times[0,1])}(f)$ for all $\Omega\ssq\kreis$
	with $\Leb(\Omega)=1$.
\end{proposition}
For the proof, we need a number of preliminary statements taken from previous
studies of pinched skew products in \cite{jaeger:2007,GroegerJaeger2013}.
First, \cite[Lemma 4.2]{GroegerJaeger2013} states that if $\alpha$ in (\ref{eq:2})
is sufficiently large, then there exist constants $\gamma,L_0,\beta,a,b>0$ and
$m\in\N$ such that the following conditions are satisfied.
\begin{eqnarray}\label{e.m}
	m & \geq & 22(1+1/\gamma) \\
    a & \geq & (m+1)^d\label{e.a}\\
    b & \leq & c\label{e.b1}\\
    b & < & d(n\omega,0) \ \textrm{for all } n\in\{1\ld m-1\}\label{e.b2}\\
    \abs{f_\theta(x)-f_\theta(y)} & \leq & \alpha^{-\gamma}\abs{x-y}
    \ \textrm{for all } \theta\in\kreis,\ x,y\in[L_0,1] \label{e.contraction}\\
    f_\theta(x) & \geq & \min\left\{L_0,ax\right\}\cdot \label{e.reference_system}
    \min\left\{1,2d(\theta,0)/b\right\} \ \textrm{for all }
    (\theta,x)\in\kreis\times[0,1]
\end{eqnarray}
It is worth mentioning that we can choose $a$ proportional to $\alpha$.
Moreover, we note that 
\begin{eqnarray}\label{e.maximal_expansion}
	\abs{f_\theta(x)-f_\theta(y)} & \leq & \alpha\abs{x-y} \ \textrm{for all } 
		\theta\in\kreis,\ x,y\in [0,1]\ ,\\
	\abs{f_\theta(x)-f_{\theta'}(x)} & \leq & \pi d(\theta,\theta') 
	    \ \textrm{for all } \theta,\theta'\in\kreis,\ x\in[0,1] \ .
\end{eqnarray}
Given any $n\in\N$, let $r_n := \frac{b}{2}a^{-\frac{n-1}{m}}$ and
$\tau_n:=n\omega$. We will need the following elementary estimate.
\begin{lemma}[\cite{jaeger:2007,GroegerJaeger2013}]
	\label{lem: consequence of diophantine w}
	Let $n\in \N$ and suppose $d(\tau_n,0)\leq \ell b\cdot a^{-i}$ for some $i>0$
	and $\ell>0$. 
	Then $n\geq\frac{a^{i/d}}{\ell^{1/d}}$.
\end{lemma}
\begin{proof}
	(\ref{e.Diophantine}) implies $c\cdot n^{-d}\leq \ell b\cdot a^{-i}$, and
	using (\ref{e.b1}) we get $n^{-d}\leq\ell a^{-i}$.
\end{proof}

In order to analyse the dynamics of $f$ on $\Phi^+$, it turns out to be crucial
that $\varphi^+$ is approximated by the so-called \emph{iterated boundary lines}
$\nfolge{\phi_n}$ of (\ref{eq:2}).
These are given by
\begin{align*}
	\varphi_n: \kreis \to [0,1];\quad\theta \mapsto f_{\theta-n\omega}^n(1) \ ,
\end{align*}
with $f^n_\theta(x):=\pi_x\circ f^n(\theta,x)=f_{\theta+(n-1)\omega}
	\circ\ldots\circ f_\theta(x)$ where $\pi_x$ is the projection onto the 
second coordinate.
Note that by the monotonicity of the maps $f_\theta$, the sequence 
$(\phi_n)_{n\in \N}$ is decreasing.
Further, as a consequence of the definition of $\phi^+$ in 
(\ref{e.upper_bounding_graph}), it can be shown easily that $\phi_n \to \phi^+$
pointwise as $n\to \infty$ \cite{GroegerJaeger2013}.
The following proposition tells us to which degree the $n$-th iterated boundary
line approximates the graph $\phi^+$.
\begin{proposition}[\cite{jaeger:2007},\cite{GroegerJaeger2013}]
	\label{properties_approx_graphs}
	Given $q\in\N$, the following holds.
	\begin{enumerate}
		 \item[(i)] $|\varphi_{n}(\theta)- \varphi_{n}(\theta')|\leq\pi\alpha^n
			d(\theta,\theta')$ for all $n\in\N$ and $\theta,\theta'\in\kreis$.
	     \item[(ii)] There exists $\lambda>0$ such that if $n\geq mq+1$ and
        	$\theta\notin\bigcup_{j=q}^n B_{r_{j}}(\tau_j)$, then
		    $|\varphi_{n}(\theta)-\varphi_{n-1}(\theta)|\leq
        		\alpha^{-\lambda(n-1)}$.
	\end{enumerate}
\end{proposition}
\begin{figure}
	\centering 
	\subfloat[]{\includegraphics[width=1.8in]{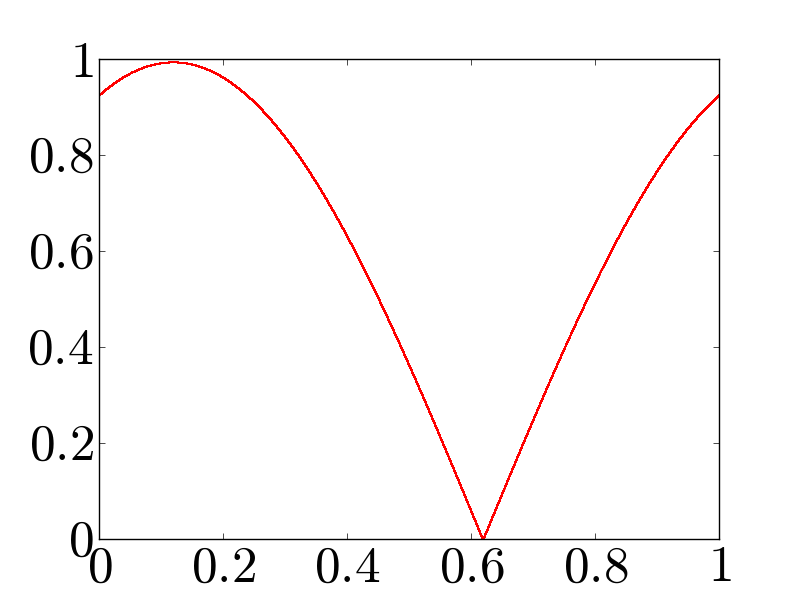}}
	\subfloat[]{\includegraphics[width=1.8in]{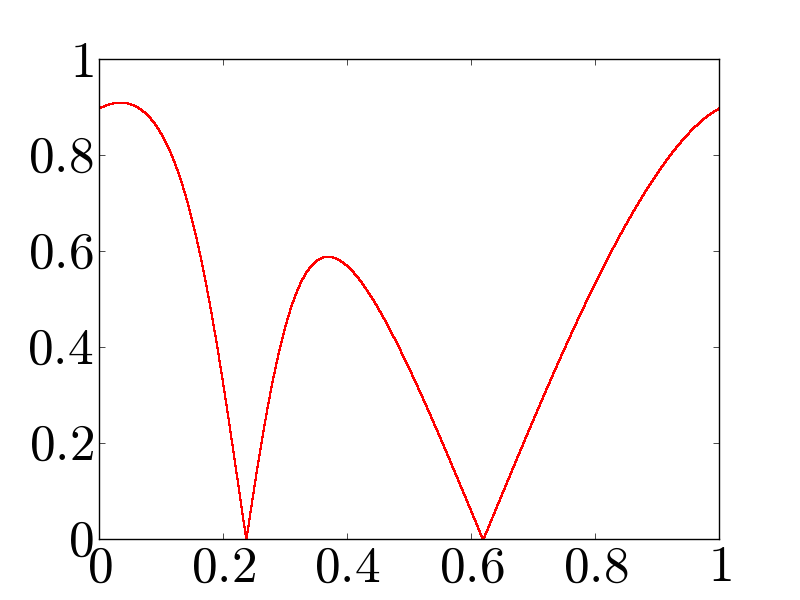}}
	\subfloat[]{\includegraphics[width=1.8in]{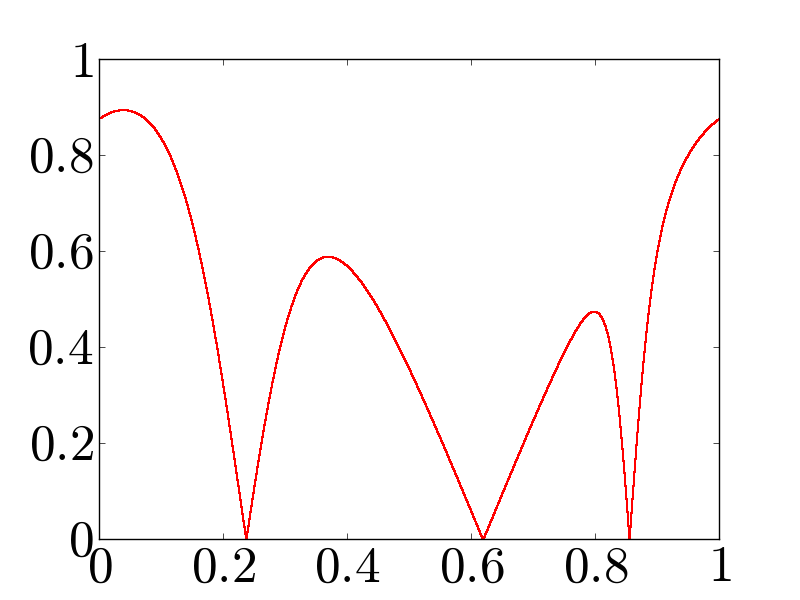}}\\
	\subfloat[]{\includegraphics[width=1.8in]{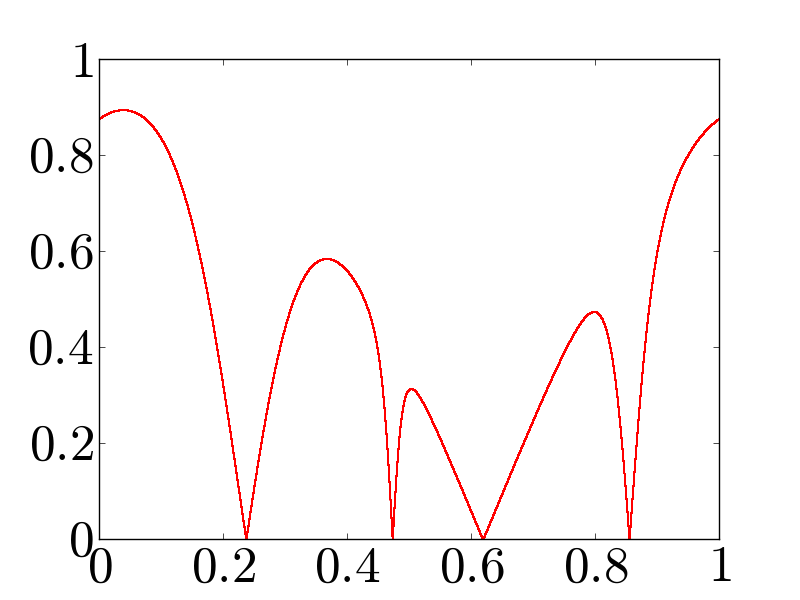}}
	\subfloat[]{\includegraphics[width=1.8in]{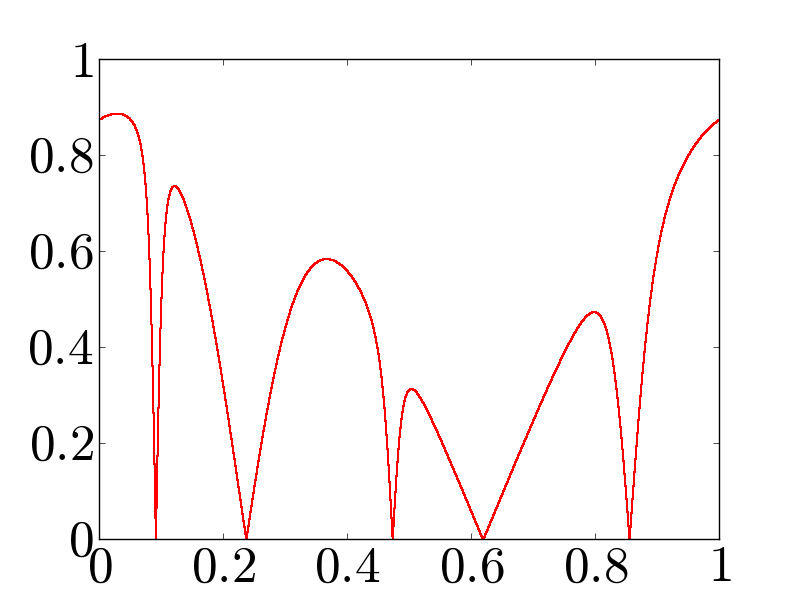}}
	\subfloat[]{\includegraphics[width=1.8in]{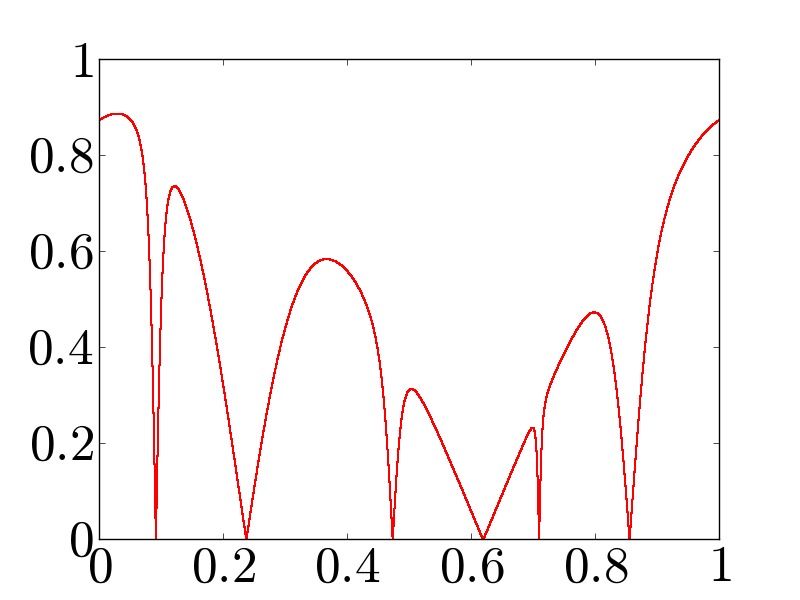}}
	\caption{The iterated boundary lines $\phi_n$ for $n=1,\ldots,6$.} 
	\label{fig: iterated boundary lines}
\end{figure}
Figure~\ref{fig: iterated boundary lines} shows the development of the iterated
upper boundary lines for $n=1\ld 6$.
As can be seen, $\phi_n$ has exactly $n$ zeros (at $\tau_1,\ldots,\tau_n$).
In order to describe the qualitative behaviour, we refer to 
$\left.\vphantom{T}\psi\right|_{B_{r_{j}}(\tau_j)}$ as the \emph{$j$-th peak} of
$\psi$ or the \emph{peak of $\psi$ around $\tau_j$}, where
$\psi\in\left\{\phi^+,\phi_j,\phi_{j+1},\ldots\right\}$.
We say the $j$-th peak is a \emph{fresh peak} if the $2r_j$-neighbourhood of
$\tau_j$ does not intersect any previous peak, that is,
$B_{2r_{j}}(\tau_j)\cap B_{r_{l}}(\tau_l)=\emptyset$ for each $1\leq
l<j$.
In the following, we label the fresh peaks by $n_1<n_2<\ldots$, that is, there
is $j\in \N$ with $n_j=l$ if and only if the $l$-th peak is fresh.
\begin{lemma}\label{lem: fresh peaks density}
	If $\alpha$ is large enough, there are infinitely many fresh peaks and they
	appear with positive density, that is,
	\begin{align*}
		\varliminf_{j\to \infty} j/n_j\ >\ 0\ .
	\end{align*}
\end{lemma}
\begin{proof}
	Let
	\begin{align*}
		N(k):= \left\{j \in \{2,\ldots,k\}\mid \textrm{there is } 1\leq l < j
	    	\text{ with } B_{2r_j}(\tau_j)\ssq B_{2 r_l}(\tau_l) \right\}.
	\end{align*}
	Thus, $N(n_j)$ contains the complement of $\left\{n_l\mid 1\leq l\leq j\right\}$.
	Further,
	\begin{align*}
		\#N(k)&\ \leq\ \sum_{l=1}^{k-1} \#\left\{j \in \{l+1,\ldots,k\}
			\mid B_{2 r_j}(\tau_j) \ssq B_{2 r_l}(\tau_l)\right\}\\
		&\stackrel{\mathclap{\textrm{Lemma~\ref{lem: consequence of diophantine w}}}}
			{\ \leq\ }\#\left\{j \in \{2,\ldots,k\}\mid B_{2 r_j}(\tau_j) \ssq B_{2r_1}
			(\tau_1)\right\}+\sum_{l=2}^{k-1} \frac{(k-l)2^{1/d}}{a^{(l-1)/md}}\\
		&\ \leq\ \#\left\{j \in \{2,\ldots,k\}\mid  B_{2 r_j}(\tau_j) \ssq B_{2 r_1}
		(\tau_1)\right\}+ 2^{1/d}k \sum_{l=2}^{k-1} a^{-(l-1)/md}\\
		&\ <\ \#\left\{j \in \{2,\ldots,k\}\mid B_{2 r_j}(\tau_j) \ssq B_{2 r_1}
			(\tau_1)\right\}+ \frac{2^{1/d}ka^{-1/md}}{1-a^{-1/md}}\ .
	\end{align*}
	Thus, for big enough $\alpha$ (and hence big enough $a$) and due to
	(\ref{e.b2}), there are infinitely many fresh peaks and
	\begin{align*}
		&\varliminf_{j\to \infty} j/n_j\ \geq\ \varliminf_{j\to \infty}
			\frac{n_j-\#N(n_j)}{n_j}\ \geq\  1-\Leb(B_{2r_1}(\tau_1))
			-\frac{2^{1/d}a^{-1/md}}{1-a^{-1/md}}\ >\ 0\ .\qedhere
	\end{align*}
\end{proof}
Since $\left(\phi_n\right)_{n\in\N}$ is monotonously decreasing and each
iterated boundary line is continuous, we know that $\phi^+$ is close to zero
in a neighbourhood of each $\tau_n$.
However, the next statement tells us that for most $\theta$ in a neighbourhood
of a fresh peak, $\phi^+$ is bigger than some threshold $\delta_0>0$.
This dichotomy is the basis for the mechanism by which we prove
Proposition~\ref{prop: ac > 0}.

\begin{lemma}\label{lem: fresh peaks height}
	Suppose $\alpha$ is large enough.
	There exist $\delta_0>0$ and a super-exponentially fast decaying sequence
	$\left(\eps_n\right)_{n\in\N}$ such that
	\begin{align*}
		\Leb\left(\left\{\theta \in B_{2r_{n_j}}(\tau_{n_j})\setminus
			B_{r_{n_j}}(\tau_{n_j})\mid \phi^+(\theta)< \delta_0\right\}\right)
		\ <\ \eps_{n_j}\ .
	\end{align*}
\end{lemma}
\begin{proof}
	Let $\ell:=m+1$.
	Since $\phi_{\ell}$ is continuous and $\phi_{\ell}(\theta)\neq 0$ for
	$\theta\notin\{\tau_1,\ldots,\tau_{\ell}\}$, there exists $\delta_0>0$ such
	that
	\begin{align}\label{eq: phi1 > L0}
		\phi_{\ell}(\theta)\ \geq\ 2\delta_0
	\end{align}
	for $\theta\notin \bigcup_{j=1}^{\ell} B_{r_j}(\tau_j)$.
	Due to (\ref{e.b2}), we have that if $\alpha$ (and hence $a$) is large enough,
	then $\kreis\setminus\bigcup_{j=1}^{\infty}B_{r_j}(\tau_j)$ is non-empty.
	Let $\theta\notin\bigcup_{j=1}^{\infty} B_{r_j}(\tau_j)$.
	For $k \in \N$ with $k>\ell$, Proposition~\ref{properties_approx_graphs} (ii)
	yields
	\begin{align*}
		\left|\phi_k(\theta)-\phi_{\ell}(\theta)\right|\ \leq\ 
		\sum_{j=\ell+1}^k\left|\phi_{j}(\theta)-\phi_{j-1}(\theta)\right|\ \leq\ 
		\sum_{j=\ell}^{k-1} \alpha^{-\lam j}\ \leq\ 
		\frac{\alpha^{-\lam \ell}}{1-\alpha^{-\lam}}\ .
	\end{align*}
	Together with equation (\ref{eq: phi1 > L0}), this gives
	\begin{align}\label{eq: phil > L0/2}
		\phi_k(\theta)\ \geq\ 
		2\delta_0-\frac{\alpha^{-\lam \ell}}{1-\alpha^{-\lam}}\ >\ \delta_0
	\end{align}
	for sufficiently large $\alpha$.

	Let $j\geq2$. Since the $n_j$-th peak is fresh, we have 
	$B_{2r_{n_j}}(\tau_{n_j})\cap \bigcup_{l=1}^{n_j-1}
	B_{r_l}(\tau_l)=\emptyset$.
	Further, Lemma~\ref{lem: consequence of diophantine w} yields that the first
	time $l>n_j$ a peak intersects $B_{2r_{n_j}}(\tau_{n_j})$ is bounded from
	below by $n_j+a^{(n_j-1)/md}$.
	Hence,
	\begin{align*}
		B_{2r_{n_j}}(\tau_{n_j})\setminus \bigcup_{l\geq 1} B_{r_l}(\tau_l)&\ =\ 
		B_{2r_{n_j}}(\tau_{n_j})\setminus\bigcup_{l\geq {n_j}} B_{r_l}(\tau_l)\\
		&\ =\ \left (B_{2r_{n_j}}(\tau_{n_j})\setminus 	
			B_{r_{n_j}}(\tau_{n_j})\right)\setminus
			\bigcup_{l\geq a^{({n_j}-1)/md}+{n_j}} B_{r_l}(\tau_l)\ .
	\end{align*}
	Note that for $n\in\N$
	\begin{align*}
		\Leb\left(\bigcup_{l\geq a^{(n-1)/md}+n} B_{r_l}(\tau_l)\right)
		&\ \leq\ b \sum_{l\geq  a^{(n-1)/md}+n}a^{-(l-1)/m}\\
		&\ =\ \frac b{1-a^{-1/m}}a^{-\left(a^{(n-1)/m}+n-1\right)/m}\ =:\ \eps_n\ .
	\end{align*}
	By means of equation \eqref{eq: phil > L0/2}, this proves the statement
	since $\phi_k\to \phi^+$ as $k\to\infty$.
\end{proof}

\begin{proof}[Proof of Proposition~\ref{prop: ac > 0}]
	Suppose $\delta<\frac{\delta_0}2$, where $\delta_0$ is chosen as in Lemma
	\ref{lem: fresh peaks height} and set $\eta_{j}:=\delta/(2\pi\alpha^{n_j})$.
	Further, let $\Omega$ be a given set of full measure.
	As $\left(\phi_n\right)_{n\in\N}$ is monotonously decreasing, Proposition 
	\ref{properties_approx_graphs} (i) shows that
	\begin{align}\label{eq: phi small close to peaks}
		\phi^+(\theta)<\delta\quad\textrm{ for all }\quad
			\theta\in B_{2\eta_j}(\tau_{n_j})\ .
	\end{align}
	By possibly going over to a subsequence (such that $n_1$ is big enough),
	Lemma~\ref{lem: fresh peaks height} yields that for each $j\in \N$ there is
	$\theta^{n_j}\in B_{2r_{n_j}}(\tau_{n_j})$ with
	\begin{align}\label{eq: phi big outside of fresh peaks}
		\Leb\left(\left\{\theta\in B_{\eta_j}
			\left(\theta^{n_j}\right)\mid\phi^+(\theta)>2\delta\right\}\right)
		\ >\ \eta_j\ .
	\end{align}
	Set $\Delta_j:=\tau_{n_j}-\theta^{n_j}$.
	By \eqref{eq: phi small close to peaks} and \eqref{eq: phi big outside of
	fresh peaks}, we have that
	\begin{align}\label{eq: measure of seperating region}
		\Leb\left(\left\{\theta\in \kreis:
			\abs{\phi^+(\theta)-\phi^+(\theta+\Delta)}>\delta \text{ for all }
			\Delta \in B_{\eta_j}(\Delta_j)\right\}\right)\ \geq\ \eta_j\ .
	\end{align}
	We denote  by $\Omega_j$ the set of such $\theta$ which visit the set
	$\{\theta\in \kreis : \abs{\phi^+(\theta)-\phi^+(\theta+\Delta)}>\delta
		\text{ for all }\Delta \in B_{\eta_j}(\Delta_j)\}$ with a frequency $\eta_j$.
	Note that by \eqref{eq: measure of seperating region} and Birkhoff's Ergodic
	Theorem, $\Leb(\Omega_j)=1$ such that
	$\tilde\Omega:=\Omega\cap\bigcap_{j\in\N}\Omega_j$ has full measure.

	Next, we choose $2^j$ points in $\Phi^+\cap \tilde \Omega\times[0,1]$ which
	are mutually $(f,\delta,\eta_j)$-separated from each other.
	Let $\theta\in\bigcap_{x\in\{0,1\}^j}\tilde\Omega-\sum_{k=1}^jx_k\Delta_k$
	and define $\theta_x=\theta+\sum_{k=1}^jx_k\Delta_k$ where
	$x=(x_1,\ldots,x_j)\in\{0,1\}^j$.
	By possibly going over to a subsequence of $\left(n_j\right)_{j\in\N}$
	(still of positive density), we may assume without loss of generality that 
	$\sum_{k=i+1}^\infty (\eta_k+|\Delta_k|)<\eta_i$ for all $i\in\N$ such that
	$d(\theta_x,\theta_y)\in B_{\eta_\ell}(\Delta_\ell)$ for distinct
	$x,y\in \{0,1\}^j$ and $\ell:=\min\{k\mid x_k\neq y_k\}\leq j$.
	By definition, we have for all $x\in\{0,1\}^j$ that $\theta_x\in\tilde\Omega$
	and hence, the set $\{(\theta_x,\phi^+(\theta_x))\mid x\in \{0,1\}^j\}$ is
	$(f,\delta,\eta_j)$-separated.
	We have thus shown
	\begin{align*}
		\varliminf_{\nu \to 0}
			\frac{\log\Sep_{\Phi^+\cap(\Omega\times[0,1])}
				(f,\delta,\nu)}{\log \nu^{-1}}
		&\ \geq\ \varliminf_{j\to\infty}
			\frac{\log\Sep_{\Phi^+\cap(\tilde\Omega\times[0,1])}
				(f,\delta,\eta_{j})}{\log \eta_{j+1}^{-1}}\\
		&\ \geq\ \varliminf_{j\to\infty}
			\frac{\log 2^j}{n_{j+1}\log\alpha-\log\delta/2\pi}
		\ =\ \frac{\log 2}{\log \alpha }\varliminf_{j\to\infty} j/n_{j+1}
		\ >\ 0
	\end{align*}
	by Lemma~\ref{lem: fresh peaks density}.  As $\varliminf_{j\to\infty} j/n_{j+1}$
	is independent of the set $\Omega$, this proves the statement.
\end{proof}

We now turn to the upper bound of the amorphic complexity.
Let $\Omega\ssq\kreis$ be the set of all $\theta\in\kreis$ such that for all
$q\in \N$
\begin{align*}
	\lim\limits_{n\to\infty}\frac{\#\left\{0\leq i\leq n-1\mid\theta+i\omega
		\in\bigcup_{j=q}^\infty B_{r_{j}}(\tau_j)\right\}}{n}\ =\ 
	\Leb\left(\bigcup_{j=q}^\infty B_{r_{j}}(\tau_j)\right).
\end{align*}
Note that Birkhoff's Ergodic Theorem yields that
$\mu_{\phi^+}(\Phi^+\cap(\Omega\times[0,1]))=\Leb(\Omega)=1$ (see Remark 
\ref{r.ac_measure}).
The upper bound on $\oac_{\Phi^+\cap(\Omega\times[0,1])}(f)$ will follow easily
from the following assertion.
\begin{lemma}\label{lem: generic points close enough->no separation}
	There exist $\kappa,c_0>0$ such that for all positive $\delta$ and small
	enough $\nu$, we have that for each $\theta,\theta'\in\Omega$ with
	$d(\theta,\theta')<\eps=\eps(\delta,\nu)=c_0\delta\nu^{\kappa m}$
	the points $(\theta,\phi^+(\theta))$ and $(\theta',\phi^+(\theta'))$ are not
	$(f,\delta,\nu)$-separated.
\end{lemma}
\begin{proof}
	Observe that there is a constant $C>0$, independent of both $\delta$ and
	$\nu$, such that for $q(\nu):=\left\lceil-C\log\nu\right\rceil$ we have
	\begin{align}\label{eq: defn of q(nu)}
		\Leb\left(\bigcup_{j=q(\nu)}^\infty B_{r_j}(\tau_j)\right)
		\ \leq\ \sum_{j=q(\nu)}^\infty a^{-(j-1)/m}\ <\ \nu/2\ .
	\end{align}
	Set $n(\nu):=mq(\nu)+1$ and assume that $n(\nu)$ is large enough (i.e.\ 
	$\nu$ is small) to guarantee
	\begin{align}\label{eq: quality of approximation after n(nu) iterations}
		\abs{\varphi^+(\theta)-\varphi_{n(\nu)}(\theta)}\ <\ \delta/4
	\end{align}
	for all $\theta\notin\bigcup_{j=q(\nu)}^\infty B_{r_{j}}(\tau_j)$
	(cf. Proposition~\ref{properties_approx_graphs}(ii)).
	Note that if $d(\theta,\theta')<\eps(\delta,\nu):=
		\delta/(4\pi\alpha^{m(-C\log \nu+1)+1})$, then
	\begin{align}\label{eq: lipschitz condition phi n(nu)}
		\abs{\varphi_{n(\nu)}(\theta)- \varphi_{n(\nu)}(\theta')}\ \leq\
		\pi\alpha^{n(\nu)}d(\theta,\theta')\ <\ \delta/4
	\end{align}
	for all $\theta,\theta'$ (cf. Proposition~\ref{properties_approx_graphs}(i)).
	Now, assume $\theta,\theta'\in\Omega$ verify
	$d(\theta,\theta')<\eps(\delta,\nu)<\delta/4$.
	By \eqref{eq: defn of q(nu)} and definition of $\Omega$, we know that
	\begin{align*}
		\lim\limits_{n\to\infty}\frac{\#\left\{0\leq i\leq n-1\mid\theta+i\omega,
			\theta'+i\omega\notin\bigcup_{j=q(\nu)}^\infty 
			B_{r_{j}}(\tau_j)\right\}}{n}
		\ >\ 1-\nu.
	\end{align*}
	Further, \eqref{eq: quality of approximation after n(nu) iterations} and
	\eqref{eq: lipschitz condition phi n(nu)} yield that 
	$\abs{\phi^+(\theta+i\omega)-\phi^+(\theta'+i\omega)}<\frac34\delta$
	whenever both $\theta+i\omega$ and $\theta'+i\omega$ are not in 
	$\bigcup_{j=q(\nu)}^\infty B_{r_{j}}(\tau_j)$.
	Hence,
	\begin{align*}
		 \varlimsup_{n\to \infty}\frac{S_n(f,\delta, (\theta,\phi(\theta)),
			(\theta',\phi(\theta')))}{n}\ <\ \nu,
	\end{align*}
	so that $(\theta,\phi(\theta))$ and $(\theta',\phi(\theta'))$ are not
	$(f,\delta,\nu)$-separated.
\end{proof}

We thus have
\begin{proposition}\label{prop: upper ac on the graph}
	Suppose $\alpha$ in (\ref{eq:2}) is sufficiently large and $\Omega$ is as in
	the previous lemma.
	Then
	\begin{align*}
		 \oac_{\Phi^+\cap(\Omega\times[0,1])}(f)\ \leq\ \kappa m,
	\end{align*}
	with $\kappa$ as in Lemma~\ref{lem: generic points close enough->no separation}.
\end{proposition}
\begin{proof}
	By the previous lemma, we know that for small enough $\nu$
	\begin{align*}
	 	&\Span_{\Phi^+\cap(\Omega\times[0,1])}(f,\delta,\nu)\ \leq\ 
		\left\lceil\frac1{\eps(\delta,\nu)}\right\rceil+1\ =\ 
		\left\lceil\frac{\nu^{-\kappa m}}{c_0\delta}\right\rceil+1\ <\ 
		2\frac{\nu^{-\kappa m}}{c_0\delta}\ .\qedhere
	\end{align*}
\end{proof}
\begin{proof}[Proof of Theorem~\ref{t.pinched_systems}]
	It is left to show the upper bound. 
	To that end, we show that there is an invariant set of full measure 
	$\tilde\Omega\ssq\Omega$ ($\Omega$ as above) such that
	\begin{align}\label{eq: fibrewise attraction}
		\lim_{n\to \infty}  \phi^+(\theta+n\omega)-f_\theta^n(x)\ = \ 0
	\end{align}
	for all $\theta\in\tilde \Omega$ and $x\in (0,1]$.
	In other words, we show that $\left.f\right|_{\tilde\Omega\times(0,1]}$ has
	the unique target property with respect to $\Phi^+\cap(\tilde\Omega\times(0,1])$.
	By means of Lemma~\ref{lem: unique target general} and
	Proposition~\ref{prop: upper ac on the graph} this yields that
	$\oac(\left.f\right|_{\tilde\Omega\times(0,1]})\leq\kappa m$ and it is easy
	to see that $\kappa m$ is in fact an upper bound for
	$\oac(\left.f\right|_{\tilde\Omega\times[0,1]})$. 

	Note that since $\phi^+$ is the upper bounding graph of the global attractor,
	we have \eqref{eq: fibrewise attraction} for all $\theta$ and
	$x\in [\phi^+(\theta),1]$.
	Now, define $\psi(\theta):=\sup\{x\in [0,\phi^+(\theta)]\mid
		\eqref{eq: fibrewise attraction}\text{ does not hold}\}$.
	Due to monotonicity, $\psi$ is an invariant graph. 
	By \cite[Proposition~3.3]{jaeger:2003}, we have that for $\Leb$-a.e.\ $\theta$
	there is $\delta(\theta)>0$ such that \eqref{eq: fibrewise attraction} holds
	for $x\in (\phi^+(\theta)-\delta(\theta),\phi^+(\theta)]$.
	Hence, $\psi$ is distinct from $\phi^+$.
	Since concavity of the fibre maps $f_\theta$ only allows for two invariant
	graphs (equivalence up to sets of measure zero, cf.\ 
	\cite[Theorem~2.1]{AnagnostopoulouJaeger2012}), this shows that $\psi=0$
	Lebesgue almost surely.
	Set $\tilde \Omega := \bigcap_{n\in\Z} n \omega +
		\{\theta \in \Omega \mid \psi(\theta)=0\}$.
	This ends the proof.
\end{proof}

\end{document}

%% file: bibstandard.tex
\newcommand{\twlog}{without loss of generality}

\newcommand{\eqand}{\ensuremath{\quad \textrm{and} \quad}}

\newcommand{\foot}{\footnote}

\newcommand{\ld}{\ensuremath{,\ldots,}}
\newcommand{\ssq}{\ensuremath{\subseteq}}
\newcommand{\smin}{\ensuremath{\setminus}}
\newcommand{\eps}{\ensuremath{\varepsilon}}
\newcommand{\wh}{\ensuremath{\widehat}}

\newcommand{\htop}{\ensuremath{h_{\mathrm{top}}}}

\newcommand{\ind}{\ensuremath{\mathbf{1}}}

\newcommand{\Leb}{\ensuremath{\mathrm{Leb}}}

\newcommand{\inte}{\ensuremath{\mathrm{int}}}

\newcommand{\diam}{\ensuremath{\mathrm{diam}}}

\newcommand{\kreis}{\ensuremath{\mathbb{T}^{1}}}

\newcommand{\torus}{\ensuremath{\mathbb{T}^2}}

\newcommand{\nfolge}[1]{\ensuremath{(#1)_{n\in\mathbb{N}}}}

\newcommand{\alphlist}{\begin{list}{(\alph{enumi})}{\usecounter{enumi}\setlength{\parsep}{2pt}
      \setlength{\itemsep}{1pt} \setlength{\topsep}{5pt}
      \setlength{\partopsep}{3pt}}}
\newcommand{\arablist}{\begin{list}{(\arabic{enumi})}{\usecounter{enumi}\setlength{\parsep}{2pt}
          \setlength{\itemsep}{1pt} \setlength{\topsep}{5pt}
          \setlength{\partopsep}{3pt}}}
\newcommand{\romanlist}{\begin{list}{(\roman{enumi})}{\usecounter{enumi}\setlength{\parsep}{2pt}
              \setlength{\itemsep}{1pt} \setlength{\topsep}{5pt}
              \setlength{\partopsep}{3pt}}}
\newcommand{\Romanlist}{\begin{list}{(\Roman{enumi})}{\usecounter{enumi}\setlength{\parsep}{2pt}
              \setlength{\itemsep}{1pt} \setlength{\topsep}{5pt}
              \setlength{\partopsep}{3pt}}}
\newcommand{\bulletlist}{\begin{list}{$\bullet$}{\setlength{\parsep}{2pt}
                \setlength{\itemsep}{1pt} \setlength{\topsep}{5pt}
                \setlength{\partopsep}{3pt}\setlength{\leftmargin}{15pt}}} 
\newcommand{\Alphlist}{\begin{list}{(\Alph{enumi})}{\usecounter{enumi}\setlength{\parsep}{2pt}
      \setlength{\itemsep}{1pt} \setlength{\topsep}{5pt}
      \setlength{\partopsep}{3pt}}}
 \newcommand{\listend}{\end{list}}

\newcommand{\T}{\ensuremath{\mathbb{T}}}

\newcommand{\N}{\ensuremath{\mathbb{N}}} 
\newcommand{\R}{\ensuremath{\mathbb{R}}}
\newcommand{\Z}{\ensuremath{\mathbb{Z}}}
\newcommand{\Q}{\ensuremath{\mathbb{Q}}}

\newcommand{\cF}{\mathcal{F}}

\newcommand{\cM}{\mathcal{M}}

\newcommand{\nLim}{\ensuremath{\lim_{n\rightarrow\infty}}}

\newcommand{\kLim}{\ensuremath{\lim_{k\rightarrow\infty}}}

\newcommand{\inergsum}{\ensuremath{\sum_{i=0}^{n-1}}}

\newcommand{\ktel}{\ensuremath{\frac{1}{k}}}

\newcommand{\ntel}{\ensuremath{\frac{1}{n}}}
